\documentclass{tacMOD}
\usepackage{amsmath, amssymb, bm}
\usepackage[all]{xy}

\title{Multiple categories of generalised quintets\\Pubbl. Mat. Univ. Genova, Preprint 607 (2017)}

\author {Marco Grandis and Robert Par\'e}

\address{Dipartimento di Matematica \\
    Universit\`a di Genova \\
    Via Dodecaneso 35 \\
    16146 - Genova, Italy \\[3ex]
	Department of Mathematics and Statistics \\
	Dalhousie University \\
	Halifax NS\\
	Canada B3H 4R2
 }

\eaddress{grandis@dima.unige.it, R.Pare@Dal.Ca}

\thanks{}



\date{}

\newtheorem{thm}{Theorem}

\def\ic{\bf\sf}
\def\ul{\underline}

\def \c {\colon}
\def \cc {\, \colon \!}
\def \q {\qquad}
\def \qq {\qquad \qquad}
\def \qqq {\qquad \qquad \qquad \qquad}
\def \sep {\,  |  \,}     
\def \bu {{\scriptscriptstyle\bullet}}

\def \ndt {\noindent}
\def \Ndt {\medskip  \noindent}
\def \skp {\medskip}

\def \for {\text{for } }

\def \inn {\text{ in } }
\def \and {\text{ and } }

\def \tilde {\raise.17ex\hbox{$\scriptstyle\mathtt{\sim}$}}   

\def \todash {\dashrightarrow}
\def \adj {\dashv}

\def \rr {\;\; \raisebox{0.6ex}{$\longrightarrow$} \hspace{-3.7ex} 
\raisebox{-0.4ex}{$\longrightarrow$} \;\;}
\def \rl {\;\; \raisebox{0.6ex}{$\longrightarrow$} \hspace{-3.7ex} 
\raisebox{-0.4ex}{$\longleftarrow$} \;\;}

\def \arv {\ar@{}}      
\def \ard {\ar@{-}}    
\def \arp {\ar@{..>}}    
\def \are {\ar@{=}}   
\def \arDb {\ar[dd] !U|\bu}   
\def \arRb {\ar[rr]|\bu}   

\def \te {\otimes}
\def \setm {\raisebox{0.3ex}{$\,\scriptstyle\setminus\,$}}
\def \ci {\!\mathbin{\raise.3ex\hbox{$\scriptscriptstyle\circ$}}\!}
\def \sub {\subset}
\def\le{\leqslant}
\def\ge{\geqslant}
\def \Sum{\raise.37ex\hbox{$\scriptstyle{\sum}$} \, }
\def \PRO{\raise.4ex\hbox{$\scriptstyle{\prod}$} \, }
\def \TE{\raise.1ex\hbox{${\bigotimes}$}}
\def \iso {\: \cong \:}

\def \id {{\rm id}}

\def \Ob {{\rm Ob}}
\def \Mor {{\rm Mor}}
\def \Hor {{\rm Hor}}
\def \Ver {{\rm Ver}}
\def \Gph {{\rm Gph}}

\def \Mlt {{\rm Mlt}}
\def \rAdj {{\rm Adj}}
\def \tv {{\rm tv}}
\def \cat {{\rm cat}}
\def \dbl {{\rm dbl}}
\def \trp {{\rm trp}}

\def \trc {{\rm trc}}

\def \cosk {{\rm cosk}}
\def \Dtrc {{\rm Dtrc}}
\def \Dcosk {{\rm Dcosk}}
\def \Dim {{\rm Dim}}
\def \Lx {{\rm Lx}}
\def \Cx {{\rm Cx}}

\def \rN {{\rm N}}
\def \lim {{\rm lim}}

\def \R {\underline{R}}
\def \uS {\underline{S}}   
   
\def \U {\underline{U}}
\def \V {\underline{V}}
\def \W {\underline{W}}

\def \al {\alpha}
\def \be {\beta}
\def \ga {\gamma}
\def \de {\delta}
\def \ep {\varepsilon}
\def \ze {\zeta}
\def \th {\vartheta}
\def \ka {\kappa}
\def \la {\lambda}
\def \si {\sigma}
\def \ph {\varphi}
\def \om {\omega}

\def \bze {{\bm \zeta}}
\def \bth {{\bm \vartheta}}
\def \bpi {{\bm \pi}}
\def \bom {{\bm \omega}}

\def \bA {{\bf A}}
\def \bB {{\bf B}}
\def \bC {{\bf C}}

\def \bCat {{\bf Cat}}
\def \bDbl {{\bf Dbl}}
\def \bAlg {{\bf Alg}}
\def \bMlt {{\bf Mlt}}
\def \bPsa {{\bf Psa}}

\def \one {{\bf 1}}
\def \two {{\bf 2}}
\def \thr {{\bf 3}}

\def \f {{\bf f}}
\def \g {{\bf g}}
\def \i {{\bf i}}

\def \h {{\bf h}}
\def \n {{\bf n}}
\def \p {{\bf p}}
\def \r {{\bf r}}
\def \s {{\bf s}}
\def \t {{\bf t}}

\def \A {\mathbb{A}}
\def \B {\mathbb{B}}
\def \C {\mathbb{C}}
\def \D {\mathbb{D}}
\def \Q {\mathbb{Q}}

\def \Dbl {\mathbb{D}{\rm bl}}
\def \Psa {\mathbb{P}{\rm sa}}
\def \Cmc {\mathbb{C}{\rm mc}}

\def \Adj {\mathbb{A}{\rm dj}}

\def \iA {{\ic A}}
\def \iB {{\ic B}}
\def \iC {{\ic C}}
\def \iD {{\ic D}}
\def \iM {{\ic M}}
\def \iP {{\ic P}}
\def \iQ {{\ic Q}}
\def \iS {{\ic S}}

\def \iCub {{\ic Cub}}
\def \iSpan {{\ic Span}}
\def \iCosp {{\ic Cosp}}

\def \iCmc {{\ic Cmc}}
\def \iInc {{\ic Inc}}
\def \iICat {{\ic ICat}}
\def \iGQ {{\ic GQ}}
\def \iAdj {{\ic Adj}}
\def \iBnd {{\ic Bnd}}
\def \iSC {{\ic SC}}
\def \iPs {{\ic Ps}}

\def \es {\varnothing}

\def \N {\mathbb{N}}

\def \cC {{\cal C}}

\def \dd {\partial}
\def \ddp {\partial^+}
\def \ddm {\partial^-}
\def \dda {\partial^\alpha}
\def \ddb {\partial^\beta}

\begin{document}

\maketitle

\begin{abstract}

We construct various multiple categories, based on generalised Ehresmann quintets. The main 
construction is a multiple category whose objects are all the `lax' multiple categories; the 
transversal arrows are their strict multiple functors while the arrows in a positive direction are 
multiple functors of a `mixed laxity', varying from the lax ones (in direction 1) to the colax 
ones (in direction $ \infty$).

\end{abstract}

\section*{Introduction}\label{Intro}

	This paper is about strict, weak and lax multiple categories, a higher dimensional extension 
of double categories that we have studied in the articles [GP6 - GP10]. The first two of them 
are about the 3-dimensional case, where {\em intercategories} (a kind of lax multiple category) 
cover and combine diverse structures like duoidal categories, Gray categories, Verity double 
bicategories and monoidal double categories.
	
	An infinite dimensional multiple category has objects, $i$-directed arrows in each direction 
$ i \in \N$, $ij$-cells of dimension two for all $ i < j$, and so on. The {\em transversal} direction 
$ i = 0 $ always has a categorical composition (strictly associative and unitary), while the composition in a {\em geometric} direction $ i > 0 $ can be weak, i.e.\ associative and unitary 
up to invertible transversal comparisons. The transversal composition has a strict interchange 
with all the geometric ones, but the latter have $ij$-interchangers, which are assumed to be 
invertible in the case of weak multiple categories; more generally, chiral multiple categories 
and intercategories have directed $ij$-interchangers, for $ i < j$. An $n$-dimensional multiple 
category has indices in the ordinal $ \n = \{0, 1,..., n-1\}$. The definitions can be found in 
\cite{GP6} for the 3-dimensional case, and in \cite{GP8} for the general one; they are briefly 
sketched in Section \ref{1.1}, below.
	
	Many examples of infinite-dimensional weak multiple categories studied in the previous 
papers are {\em of cubical type}: loosely speaking, all the positive directions are equivalent. 
For instance this happens for all weak multiple categories $ \iSpan(\bC) $ and $ \iCosp(\bC)$, 
of `cubical' spans and cospans over a category $ \bC $ \cite{GP8}. Even many chiral examples, 
like the chiral triple category $ \iSC(\bC) $ of spans and cospans over $ \bC $ (and its infinite 
dimensional extensions) only have two kinds of arrows in positive directions: either spans or 
cospans.
	
	Here we construct various (strict) multiple categories of highly non-cubical character, 
where the arrows in each direction are of different kinds.
	
	The starting point is Ehresmann's double category of quintets $ \Q(\bC) $ over a 
2-category $ \bC$, whose horizontal and vertical maps are the maps of $ \bC$, with 
double cells defined by 2-cells of $ \bC$
%
    \begin{equation} \begin{array}{c} 
    \xymatrix  @C=12pt @R=5pt
{
~A~   \ar[rr]^-{f}   \ar[dd]_-{u}   &   \ar@/^/[ld]^\ph   &   ~B~   \ar[dd]^-{v}
\\ 
&&&&&&   \ph\c vf \to gu\c A \to D.
\\ 
~C~   \ar[rr]_-{g}    &&   ~D~
}
    \label{0.1} \end{array} \end{equation}

	Their horizontal and vertical compositions are obvious. This construction can be easily 
extended to a multiple category $ \iQ(\bC) $ of {\em higher quintets of} $ \bC $ (see Section
 \ref{1.2}). It is of cubical type, but gives a mould in which we can cast various multiple 
 categories `of generalised quintets', no longer of cubical type.

\Ndt  {\em Outline.} In Section 1, after constructing the multiple category $ \iQ(\bC) $ 
{\em of higher quintets} over a 2-category $ \bC$, we define the notion of a multiple 
category of generalised quintets, or - more particularly - of quintet type (Section \ref{1.4}).

	Section 2 contains our main result: the construction of the multiple category $\iCmc $ 
of chiral multiple categories, indexed by the ordinal $ \bom + \one = \{0, 1,..., \infty \}$. Here 
the transversal arrows are strict multiple functors while, in direction $ p $ (for $ 1 \le p \le \infty $), 
the $p$-{\em morphisms} are `multiple functors of mixed laxity', including the lax ones (in 
direction 1) and the colax ones (in direction $ \infty $). Similar frameworks are concerned with 
intercategories, and the $n$-dimensional case.
	
	The next two sections construct a multiple category $ \iGQ(\cC) $ {\em of generalised 
quintets} over a sequence $ \cC $ of 2-categories and 2-functors. In particular, a 
2-category $ \bC $ gives a multiple category $ \iAdj(\bC) $ where an $i$-directed arrow is a 
chain $ u_0 \adj u_1 ... \adj u_i $ of consecutive adjunctions in $ \bC$, and a multiple category 
$ \iBnd(\bC) $ where an $i$-directed arrow is a bundle $ u = (u_0,..., u_i) $ of parallel arrows 
of $ \bC$.
	
	Finally in Section 5 we construct a triple category of pseudo algebras, for a 2-monad; it is 
again of quintet type. The Theorems \ref{5.7} and \ref{5.8} show how algebras and normal 
pseudo algebras of graphs of categories are related to strict and weak double categories.

\Ndt  {\em Literature.} Strict double and multiple categories were introduced and studied by 
C. Ehresmann and A.C. Ehresmann \cite{Eh, BE, EE1, EE2, EE3}. Strict cubical categories can 
be seen as a particular case of multiple categories (as shown in \cite{GP8}); their links with strict 
$\om$-categories are made clear in \cite{BrM, ABS}. The theory of weak double categories (or 
pseudo double categories) is analysed in our series [GP1 - GP4], and in other papers like 
\cite{DP1, DP2, DPP1, DPP2, Fi, FGK1, FGK2, Ga, GP11, Ko1, Ko2, Ni, P1, P2}. For weak cubical 
categories see \cite{G1, G2, G3, GP5}.

\Ndt  {\em Notation.} We mainly follow the notation of [GP8 - GP10]. The symbol $ \sub $ denotes 
weak inclusion. Categories and 2-categories are generally denoted as $ \bA, \bB...$; weak 
double categories as $ \A, \B...$; weak or lax multiple categories as $ \iA, \iB...$ More specific 
points are recalled below, in Section 1.

\Ndt  {\em Acknowledgements}. This work was partially supported by GNSAGA, a research group of INDAM (Istituto Nazionale di Alta Matematica), Italy.

	
\section{Higher quintets}\label{s1}

		After constructing a multiple category $ \iQ(\bC) $ {\em of (higher) quintets} over a 
2-category $ \bC$, we define multiple categories of generalised quintets, and - more 
particularly - of quintet type.

	In a 2-category the vertical composition of 2-cells is written as $\ph \te \psi $, in 
diagrammatic order; the whisker composition of arrows and cells by juxtaposition or dots.

\subsection{Notation}\label{1.1}
The definitions of weak and chiral multiple categories can be found in \cite{GP8}, or - 
briefly reviewed - in \cite{GP9}, Section 1. Here we only give a sketch of them, while recalling 
the notation we are using.

	The two-valued index $ \al $ (or $ \be$) varies in the set $ 2 = \{0, 1\}$, also written as 
$ \{-, +\}$.

	A {\em multi-index} $ \i $ is a finite subset of $ \N$, possibly empty. Writing $ \i \sub \N $ 
it is understood that $ \i $ is finite; writing $ \i = \{i_1,..., i_n\} $ it is understood that $ \i $ has 
$ n $ distinct elements, written in the natural order $ i_1 < i_2 < ... < i_n$; the integer $ n \ge 0 $ 
is called the {\em dimension} of $ \i$. We write:
%
    \begin{equation} \begin{array}{ccc}
\i j  =  j\i  =  \i \cup \{j\}      & \qq &  (\for  j \in \N \setm \i),
\\[5pt]
 \i|j  =  \i \setm \{j\}   &&  (\for  j \in \i).
    \label{1.1.1} \end{array} \end{equation}

	For a weak multiple category $ \iA$, the set of $\i$-cells $ A_\i  $ is written as 
$ A_*$, $ A_i$, $A_{ij} $ when $ \i  $ is $\es$, $\{i\}$ or $ \{i, j\} $ respectively. Faces and 
degeneracies, satisfying the {\em multiple relations}, are denoted as
%
    \begin{equation}
\dda_j \c X_\i  \to  X_{\i |j},    \q    e_j\c  X_{\i |j} \to  X_\i    \qq    (\for \al = \pm, \;  j \in \i).
    \label{1.1.2} \end{equation}

	The {\em transversal direction} $ i = 0 $ is set apart from the positive, or {\em geometric}, 
directions. For a {\em positive multi-index} $ \i  = \{i_1,..., i_n\} \sub \N^* = \N \setm \{0\}$, the 
{\em augmented multi-index} $ 0\i  = \{0, i_1,..., i_n\} $ has dimension $ n+1$, but both $ \i  $ 
and $ 0\i  $ have {\em degree} $ n$. An $\i$-cell $ x \in A_\i  $ of $ \iA $ is also called an 
$\i$-{\em cube}, while a $0\i $-cell $ f \in A_{0\i} $ is viewed as an $\i$-{\em map} 
$ f\c x \to_0 y$, where $ x = \ddm_0 f $ and $ y = \ddp_0 f$. Composition in direction 0 is 
categorical (and generally realised by ordinary composition of mappings); it is written as 
$ gf = f +_0 g$, with identities $ 1_x = \id (x) = e_0(x)$.
	
	The {\em transversal category} $\, \tv_\i (\iA)$ consists of the $\i$-cubes and $\i$-maps of 
$ \iA$, with transversal composition and identities. Their family forms a multiple object in 
$ \bCat$, indexed by the positive multi-indices.
	
	Composition of $\i$-cubes and $\i$-maps in a {\em positive} direction $ i \in \i  $ (often 
realised by pullbacks, pushouts, tensor products, etc.) is written in additive notation
%
    \begin{equation} \begin{array}{lr}
x +_i y   &   (\ddp_ix = \ddm_iy),
\\[5pt]
f +_i g\c x +_i y \to x' +_i y'   ~~~  &    (f\c x \to x', \,  g\c y \to y', \, \ddp_if = \ddm_ig).
    \label{1.1.3} \end{array} \end{equation}

	The transversal composition has a strict interchange with each of the positive operations.
The latter are {\em categorical} and satisfy the {\em interchange law} up to 
transversally-invertible comparisons (for $ 0 < i < j$, see \cite{GP8}, Section 3.2)
%
    \begin{equation} \begin{array}{l}
\la_ix\c (e_i\ddm_ix) +_i x \to_0 x   \qqq \qq~  (\text{{\em left i-unitor}}),
\\[5pt]
\rho_ix\c x +_i (e_i\ddp_ix) \to_0 x  \qqq \qq  (\text{{\em right i-unitor}}),
\\[5pt]
\ka_i(x, y, z)\c x +_i (y +_i z) \to_0 (x +_i y) +_i z   \qq\q  (\text{{\em i-associator}}),
\\[5pt]
\chi _{ij}(x, y, z, u)\c (x +_i y) +_j (z +_i u) \to_0 (x +_j z) +_i (y +_j u)
\\[3pt]
\qqq \qqq \qq\q  (\text{{\em ij-interchanger}}).
    \label{1.1.4} \end{array} \end{equation}

	The comparisons are natural with respect to transversal maps; $ \la_i, \rho_i $ and 
$ \ka_i $ are special in direction $ i $ (i.e.\ their $i$-faces are transversal identities) while 
$ \chi_{ij} $ is special in both directions $ i, j$; all of them commute with $ \dda_k $ for 
$ k \neq i $ (or $ k \neq i, j $ in the last case). Finally the comparisons must satisfy various 
conditions of coherence, listed in \cite{GP8}, Sections 3.3 and 3.4.
	
	More generally for a {\em chiral multiple category} $ \iA $ the $ij$-interchangers 
$ \chi_{ij} $ are not assumed to be invertible (see \cite{GP8}, Section 3.7).
	
	Even more generally, in an {\em intercategory} we also have $ij$-interchangers 
$ \mu_{ij}, \de_{ij}, \tau_{ij}$ involving the units; this extension is studied in \cite{GP6, GP7} for 
the 3-dimensional case, the really important one. The infinite dimensional case is introduced 
in \cite{GP8}, Section 5, but lacks examples of interests and has a marginal position in 
[GP8 - GP10], as well as here.
	
	Lax multiple functors and their transversal transformations are defined in \cite{GP8}, Section 3.9.

\subsection{A framework of higher quintets}\label{1.2}
We start from a 2-category $ \bC $ and construct a multiple category $ \iM = \iQ(\bC) $ 
{\em of (higher) quintets} over $ \bC$, extending the double category $ \Q(\bC) $ of quintets 
introduced by C. Ehresmann (see \cite{GP1}, Section 1.3).

\Ndt  (a) The objects of $ \iM $ are those of $ \bC$; in every direction $ i \ge 0$, an $i$-cell 
$ f\c X \to_i Y $ is a $\bC$-morphism. They form the category $\cat_i(\iM) $ underlying $ \bC$.

\Ndt  (b) In dimension 2, an $ij$-cell (for $ 0 \le i < j$) is a {\em quintet} of $ \bC$, consisting of 
four morphisms and a 2-cell $ \ph$
%
    \begin{equation} \begin{array}{c} 
    \xymatrix  @C=12pt @R=6pt
{
~A~   \ar[rr]^-{r}   \ar[dd]_-{u}   &   \ar@/^/[ld]^\ph   &   ~B~   \ar[dd]^-{v}  &&&&
\bu   \ar[r]^{i}   \ar[d]^{j}  &
\\ 
&&&&   \ph\c vr \to su\c A \to D.   &&
\\ 
~C~   \ar[rr]_-{s}    &&   ~D~
}
    \label{1.2.1} \end{array} \end{equation}

	These cells have two obvious composition laws, in directions $ i $ and $ j$, and form a 
double category $ \dbl_{ij}(\iM)$; it is the same as the double category $ \Q(\bC)$, `displayed' 
in directions $ i $ and $ j$.

\Ndt  (c) In dimension 3 an $ijk$-cell (for $ 0 \le i < j < k$) is a `commutative cube' $ \Pi $ of 
quintets
%
    \begin{equation}
    \begin{array}{c}  \xymatrix  @C=8pt   @R=15pt
{
~\bu~   \ar[rrr]^-r    \ar[dd]_x   \ar[rd]^>>>>u  &&&
~\bu~     \ar[rd]^-v   &&&&
~\bu~    \ar[rrr]^r   \ar[dd]_-x  &&&
~\bu~    \ar[dd]_-{x'}   \ar[rd]^-v   
\\ 
& ~\bu~    \ar[rrr]^-s   \ar[dd]^y \arv[rru]|-\ph  &&&   
~\bu~  \ar[dd]^-{y'}  &&&&&&&   ~\bu~  \ar[dd]^-{y'}   &&&
~\bu~  \ar[rr]^>>>>i   \ar[rd]^>>>>j  \ar[d]_<<<<{k\,}  &&
\\ 
~\bu~  \ar[rd]_{u'}  \arv[ru]|-\om  &&&&&&&
~\bu~ \ar[rrr]^-{r'}    \ar[rd]_{u'}    \arv[rrruu]|-\pi  \arv[rrrrd]|-\psi  &&&
~\bu~  \ar[rd]_<{v'}  \arv[ru]|-\ze   &&&&&
\\ 
&   ~\bu~    \ar[rrr]_-{s'}  \arv[rrruu]|-\rho  &&&  ~\bu~   &&&
&   ~\bu~    \ar[rrr]_-{s'}    &&&  ~\bu~ 
}
    \label{1.2.2} \end{array} \end{equation}

	More precisely, we have six quintets
    \begin{equation*} \begin{array}{llr}
\ph\c vr \to su,   \q  &   \psi \c v'r' \to s'u'   \q~~   &   \text{(two $ij$-cells, the faces  $\dda_k\Pi$),}
\\[5pt]
\pi\c x'r \to r'x,   &   \rho\c y's \to s'y   &   \text{(two $ik$-cells, the faces $\dda_j \Pi$),}
\\[5pt]
\om\c yu \to u'x,   &  	\ze\c y'v \to v'x'   &   \text{(two $jk$-cells, the faces $ \dda_i \Pi$).}
    \label{} \end{array} \end{equation*}
which must satisfy the following commutativity relation, in $ \bC$
    \begin{equation}
y'\ph \te \rho u \te s'\om  =  \ze r \te v'\pi \te \psi x\c  y'vr \to s'u'x.
    \label{1.2.3} \end{equation}

	The compositions in directions $ i, j, k $ amount to compositions of faces in the double 
categories $ \dbl_{ij}(\iM)$, $ \dbl_{ik}(\iM) $ and $ \dbl_{jk}(\iM)$. We have thus a triple 
category $ \trp_{ijk}(\iM)$.
	
\Ndt  (d) Every cell of dimension $ n > 3 $ is an $n$-dimensional cube whose 2-dimensional 
faces are quintets, under the condition that each 3-dimensional face in direction $ ijk $ be an 
$ijk$-cell, as defined above.

\subsection{Coskeletal dimension}\label{1.3}
We want to express point (d) above saying that the multiple category $ \iQ(\bC) $ has 
{\em coskeletal dimension} 3, at most. Loosely speaking, this means that each cell of higher 
dimension is `bijectively determined' by its boundary.

	To formalise this property, we say that an $n$-{\em dimensional} 
$\N$-{\em multiple category} $ \iA $ is the `same' as a multiple category, but based on 
multi-indices $ \i  = \{i_1,..., i_k\} \sub \N $ of dimension $ k \le n$.

	There is an obvious truncation functor
    \begin{equation}
\Dtrc_n\c \bMlt \to \Dim_n\bMlt,
    \label{1.3.1} \end{equation}
defined on the category of (small) multiple categories (and strict multiple functors), with values 
in the category of $n$-dimensional $\N$-multiple categories. Its right adjoint
    \begin{equation}
\Dcosk_n\c \Dim_n\bMlt \to \bMlt,
    \label{1.3.2} \end{equation}
extends an $n$-dimensional $\N$-multiple category $ \iA$ by adding cells of every higher 
dimension, with their faces, operations and degeneracies.

	The higher cells are inductively defined as follows. If we have already defined the cells 
up to dimension $ k \ge n$, an $\i$-cell of dimension $ k+1 $ is a {\em face-consistent} 
$\i$-{\em family} $ x = (x^\al_i ) $ of $k$-dimensional cells, where
    \begin{equation} \begin{array}{lr}
x^\al_i  \; \text{ is an $\i |i$-cell}	    &    (\for  i \in \i   \and  \al = \pm),
\\[5pt]
\ddb_j x^\al_i   =  \dda_i x^\be_j     \qq   &    (\for  i \neq j  \inn  \i  \and  \al, \be = \pm).
    \label{1.3.3} \end{array} \end{equation}

	Of course the new faces are defined letting $ \dda_i (x) = x^\al_i $. Then the new 
operations and degeneracies are determined by the fact that they must be consistent with 
faces, i.e.\ satisfy the following conditions (for $ j \neq i$):
    \begin{equation} \begin{array}{c}
\ddm_i (x +_i y)  =  \ddm_i (x),  \q   \ddp_i (x +_i y)  =  \ddp_i (y),
\\[5pt]
\dda_j (x +_i y)  =  \dda_j (x) +_i \dda_j (y),
\\[5pt]
 \dda_i (e_ix)  =  x,   \qq    \dda_j (e_ix)  =  e_i\dda_j (x).
    \label{1.3.4} \end{array} \end{equation}

	Now we say that the multiple category $ \iA $ has {\em coskeletal dimension} $\le n $ if 
it is isomorphic to $\Dcosk_n\Dtrc_n(\iA)$, so that all its cells of higher dimension can be 
reconstructed from those of dimension up to $ n$, as described above. The 
{\em coskeletal dimension} of $ \iA $ is the least such natural number $ n$.
	
	Similar definitions can be given for weak or lax multiple categories. 
	
	The coskeletal dimension of the multiple category $ \iQ(\bC) $ is 3 at most, and can be less. 
For instance, if $ \bC $ is a preordered category (viewed as a locally preordered 2-category), 
then each cube of quintets \eqref{1.2.2} commutes, and is `bijectively determined' by its 
boundary, so that the coskeletal dimension of $ \iQ(\bC) $ is at most 2; moreover, if the preorder 
of $ \bC $ is codiscrete, even a quintet is `bijectively determined' by its boundary, and the 
coskeletal dimension is at most 1; finally, if $ \bC $ is a codiscrete category with codiscrete 
preorder, the coskeletal dimension of $ \iQ(\bC) $ is 0.
	
	In \cite{GP8}, Section 2.7, we have considered a different truncation functor 
$ \trc_n\c \bMlt \to \bMlt_\n$, taking values in the category of $\n$-multiple categories, 
having multi-indices $ {\i} \sub \n = \{0,..., n -1\}$. This functor and its right adjoint $ \cosk_n $ 
would give another notion of dimension, which seems to be of little interest.

\subsection{Multiple categories of generalised quintets}\label{1.4}
A {\em multiple category of generalised quintets} will be a multiple category $\iA$ equipped 
with a `forgetful' multiple functor $ U\c \iA \to \iQ(\bC) $ with values in the multiple category of 
quintets over some 2-category $ \bC$.

	More particularly we say that $ \iA $ is {\em of quintet type} over $ \bC $ when $ U $ 
satisfies the following condition of `2-dimensional faithfulness': 

\Ndt  (i) each cell $ x $ of $\iA$ of dimension 2 or higher is determined by its boundary and by 
the cell $ U(x) $ of $ \iQ(\bC)$.

\skp

	In the general case above, $U$ is just expected to capture properties of $\iA$, motivating 
our interest. Much in the same way as we may be interested in `forgetful' functors of ordinary 
categories, even if not faithful.

\subsection{Multiple categories of cubical quintets}\label{1.5}
Extending our construction of the multiple category $ \iQ(\bC) $ of quintets over a 2-category 
$ \bC$, one can start from a (strict) $n$-category $ \bC $ and form a multiple category 
$ \iQ^n(\bC)$, of coskeletal dimension $ \le n+1$, so that $ \iQ^2(\bC) $ is the previous case 
of Section \ref{1.2} while $ \iQ^1(\bC) = \iCub(\bC) $ is the multiple category of commutative cubes 
over a category (see \cite{GP8}, Section 3.5(a)).

	We only sketch the construction of $ \iQ^3(\bC)$, for a 3-category $ \bC$. For all cells of 
dimension 0, 1 and 2 we proceed as in Section \ref{1.2}, points (a), (b); then we go on as follows.

\Ndt (c$'$) In dimension 3 an $ijk$-cell $ \Pi $ (for $ 0 \le i < j < k$) is a cubical diagram of 
2-dimensional cells (its faces), as in diagram \eqref{1.2.2}, inhabited by a 3-cell of 
$ \bC $ (still written as $ \Pi$, by abuse of notation)
    \begin{equation}
\Pi\c  y'\ph \te \rho u \te s'\om \to \ze r \te v'\pi \te \psi x\c  y'vr \to s'u'x.
    \label{1.5.1} \end{equation}

	The compositions in directions $ i, j, k $ are defined by `pasting inhabited cubes' in 
$ \bC$; we get a triple category $ \trp_{ijk}(\iQ^3(\bC))$.

\Ndt  (d$'$) A cell of dimension 4 is a face-consistent family of 3-dimensional cells that forms a 
`commutative 4-dimensional cube'. Every cell of dimension $ n > 4 $ is an $n$-dimensional
 cube whose 3-dimensional faces are as in (c$'$), under the condition that each 4-dimensional 
 face be a commutative 4-cube.

	
\section{The multiple category of chiral multiple categories}\label{s2}

	In \cite{GP10}, Section 2, we have constructed the strict {\em double} category $ \Cmc $ of chiral multiple categories, with lax and colax multiple functors, and suitable double cells `of quintet type'. This construction was the basis of our definition of colax/lax adjunctions between chiral multiple categories.

	In \cite{GP6}, Section 6, we have constructed the strict {\em triple} category $ \iICat $ of 3-dimensional intercategories, whose arrows are `mixed-laxity functors': the lax triple functors in direction 1, the colax ones in direction 3 and the intermediate {\em colax-lax morphisms} in direction 2.

	We extend now these structures, forming a (strict) multiple category $ \iCmc $ of chiral 
multiple categories, indexed by the ordinal $ \bom + \one = \{0, 1,..., \infty \}$. The transversal 
arrows, or 0-{\em morphisms}, are strict multiple functors. In direction $ p $ 
(for $ 1 \le p \le \infty $), the $p$-{\em morphisms} are `functors of mixed laxity', varying from 
the lax functors (in direction 1) to the colax ones (in direction $ \infty $).

	As to notation, a chiral multiple category $ \iA $ is a multiple object of ordinary categories 
$\tv_\i (\iA) $ indexed by positive multi-indices $ \i = \{i, j, k...\} \sub \N^*$. On the other hand 
$ \iCmc $ will be indexed by `extended' positive multi-indices 
$ \p =\{p, q, r...\} \sub \{1, 2,..., \infty \}$.

	Let us note that, if we restrict $ \iCmc $ to the weak multiple categories of cubical type 
(see \cite{GP8}), we still have a multiple category of {\em non}-cubical type, with different 
kinds of arrows in each direction.

\subsection{Mixed-laxity functors}\label{2.1}
In degree 0, the objects of $ \iCmc $ are the (small) chiral multiple categories, and the 
transversal arrows (or 0-morphisms) are the strict multiple functors $ F\c \iA \to_0 \iB$.

	In degree 1 and direction $ p $ (for $ 1 \le p \le \infty $), a {\em $p$-morphism} 
$ R\c \iA \to_p \iB $ between chiral multiple categories will be a {\em mixed-laxity functor} 
which is colax in all positive directions $ i < p $ and lax in all directions $ i \ge p$. In particular, 
this is a lax functor for $ p = 1 $ and a colax functor for $ p = \infty $.

	Basically, $ R $ has components $ R_\i  = \tv_\i (R)\c \tv_\i (\iA) \to \tv_\i (\iB)$, for all 
positive multi-indices $ \i$, that are ordinary functors and commute with faces: 
$ \dda_i .R_\i  = R_{\i|i}.\dda_i  $ (for $ i \in \i$).

	Moreover $ R $ is equipped with $i$-special {\em comparison} $\i$-maps $ \R_i$  
(for  $t  \in  A_{\i|i}$ and $x, y$ $i$-consecutive in $A_\i$), either in the lax direction for $ i \ge p$
    \begin{equation}
\R_i(t)\c e_iR(t) \to_0 R(e_it),    \q    \R_i(x, y)\c R(x) +_i R(y) \to_0 R(x +_i y),
    \label{2.1.1} \end{equation}
\Ndt  or in the colax direction for $ 0 < i < p$
    \begin{equation}
\R_i(t)\c R(e_it) \to_0 e_iR(t),    \q    \R_i(x, y)\c R(x +_i y) \to_0 R(x) +_i R(y).
    \label{2.1.2} \end{equation}

	All these comparisons must commute with faces (for  $j \neq i$  in $\i$)
    \begin{equation}
\dda_j \R_i(t)  =  \R_i(\dda_j t),	  \qq   \dda_j \R_i(x, y)  =  \R_i(\dda_j x, \dda_j y).
    \label{2.1.3} \end{equation}

	Moreover they have to satisfy the axioms of naturality and coherence (see \cite{GP8}, 
Section 3.9), either in the lax form (lmf.1-4) for $i \ge p$, or in the transversally dual form 
- say (cmf.1-4) - for $ i < p$.

	Finally there is an axiom of coherence with the interchanger $ \chi_{ij} $ $(0 < i < j) $ 
which has three forms (where (a) corresponds to (lmf.5), (c) corresponds to its dual (cmf.5) 
and (b) is an intermediate case):

\Ndt  (a) for $ p \le i < j $ (so that $ R $ is $i$- and $j$-lax), we have commutative diagrams 
of transversal maps:
    \begin{equation} \begin{array}{c} 
    \xymatrix  @C=25pt @R=15pt
{
~(Rx +_i Ry) +_j (Rz +_i Ru)~   \ar[r]^-{\chi_{ij} R}   \ar[d]_-{\R_i +_j \R_i}    &
~(Rx +_j Rz) +_i (Ry +_j Ru)~   \ar[d]^-{\R_j +_i \R_j}
\\ 
~R(x +_i y) +_j R(z +_i u)~   \ar[d]_-{\R_j}    &   
~R(x +_j y) +_i R(z +_j u)~   \ar[d]^-{\R_i}
\\ 
~R((x +_i y) +_j (z +_i u))~   \ar[r]_-{R\chi_{ij}}    &   ~R((x +_j z) +_i (y +_j u))~
}
    \label{2.1.4} \end{array} \end{equation}

\Ndt (b) for $ 0 < i < p \le j $ (so that $ R $ is $i$-colax and $j$-lax), we have commutative diagrams:
    \begin{equation} \begin{array}{c} 
    \xymatrix  @C=35pt @R=45pt  @H=15pt   @!0
{
& (Rx +_i Ry) +_j (Rz +_i Ru)   &&    \ar[r]^-{\chi_{ij} R}    &&&
(Rx +_j Rz) +_i (Ry +_j Ru)   \ar[rd]^-{\R_j +_i \R_j}
\\
R(x +_i y) +_j R(z +_i u)  \ar[ru]^-{\R_i +_j \R_i}   \ar[rd]_-{\R_j}   &&&&&&&       
R(x +_j y) +_i R(z +_j u)
\\
&R((x +_i y) +_j (z +_i u))    &&    \ar[r]_-{R\chi_{ij}}    &&&    
R((x +_j z) +_i (y +_j u))    \ar[ru]_-{\R_i}
}
    \label{2.1.5} \end{array} \end{equation}

\Ndt  (c) for $ 0 < i < j < p $ (so that $ R $ is $i$- and$j$-colax), we have commutative diagrams as in \eqref{2.1.4}, with all vertical arrows reversed.

\skp

	The composition of $p$-morphisms is easily defined: their comparisons are separately composed.
	
	Finally, a transversal map $ (F, G)\c R \to_0 S $ of $p$-arrows will be a commutative square
    \begin{equation} \begin{array}{c} 
    \xymatrix  @C=15pt @R=5pt
{
~\bu~   \ar[rr]^-{F}   \arDb_-{R}   &&   ~\bu~   \arDb^-{S}  &&&&
\bu   \ar[r]^{0}   \ar[d]^{p}  &
\\ 
&  =  &&&  SF  =  GR   &&
\\ 
~\bu~   \ar[rr]_-{G}    &&   ~\bu~
}
    \label{2.1.6} \end{array} \end{equation}
with strict functors $ F, G $ and $p$-morphisms $ R, S$. Commutativity means that 
$ SF = GR $ {\em as} $p$-morphisms, including comparisons.

	(Let us recall that, as already remarked in \cite{GP6}, the `lax-colax' case makes no sense: 
modifying diagram \eqref{2.1.4} by reversing all arrows $ \R_j$ would lead to a diagram where 
no pairs of arrows compose.)

	We have thus defined the double category $ \dbl_{0p}(\iCmc) $ of chiral multiple 
categories, strict functors and $p$-morphisms.

\subsection{Two-dimensional cubes}\label{2.2}
To define a $pq$-cube (for $ 1 \le p < q \le \infty $) we have to adapt the axioms of transversal 
transformation (again in \cite{GP8}, Section 3.9).

\skp

	A $pq$-{\em cube} $ \ph\c (U \; ^R_S  \; V) $ will be a `generalised quintet' consisting of 
two $p$-morphisms $ R, S$, two $q$-morphisms $ U, V$, together with - roughly speaking - 
a `transversal transformation' $ \ph\c VR \todash SU$
%
    \begin{equation} \begin{array}{c} 
    \xymatrix  @C=15pt @R=5pt
{
~\iA~   \arRb^-{R}   \arDb_-{U}   &   \ar@{-->}@/^/[ld]^-\ph   &   ~\bu~   \arDb^-{V}  &&&&
\bu   \ar[r]^{p}   \ar[d]^{q}  &
\\
&&&&   \ph\c VR \todash SU.  &&
\\ 
~\bu~   \arRb_-{S}    &&   ~\iB~ 
}
    \label{2.2.1} \end{array} \end{equation}

	Again (as in \cite{GP10}, Section 2, for the double category $ \Cmc$) this is an abuse 
of notation since there are no composites $ VR $ and $ SU $ in our structure: the coherence 
conditions of $ \ph $ are based on the four morphisms $ R, S, U, V $ and all their comparison 
maps. Precisely, the cell $ \ph $ consists of a face-consistent family of transversal maps in 
$ \iB$
    \begin{equation} \begin{array}{lr}
\ph(x) = \ph_\i (x)\c VR(x) \to_0 SU(x),    \q   &    \text{(for every  $\i$-cube $x$ of $\iA$),}
\\[5pt]
\dda_i .\ph_\i   =  \ph_{\i|i}.\dda_i    &    (\for  i \in \i),
    \label{2.2.2} \end{array} \end{equation}
so that each component $ \ph_\i \c V_\i R_\i  \to S_\i U_\i \c \tv_\i (\iA) \to \tv_\i (\iB) $ is a 
natural transformation of ordinary functors:

\Ndt  (nat) for all $ f\c x \to_0 y $ in $ \iA$, we have a commutative diagram of transversal maps in $ \iB$
    \begin{equation} \begin{array}{c} 
    \xymatrix  @C=10pt @R=8pt
{
~VR(x)~   \ar[rr]^-{\ph x}   \ar[dd]_-{VRf}   &&   ~SU(x)~   \ar[dd]^-{SUf}
\\ 
&  = 
\\ 
~VR(y)~   \ar[rr]_-{\ph y}    &&   ~SU(y)~
}
    \label{2.2.3} \end{array} \end{equation}

	Moreover $ \ph  $ has to satisfy the following coherence conditions (coh.a), (coh.b), 
(coh.c) with the comparisons of $ R, S, U, V$, for a degenerated cube 
$ e_i(t)$ (with $t \in A_{\i|i}$) and a composite $ z = x +_i y $ in $ A_\i $.

\Ndt  (coh.a) If $ p < q \le i $ (so that $ R, S, U, V $ are lax in direction $ i$), we have commutative diagrams:
    \begin{equation} \begin{array}{c} 
    \xymatrix  @C=28pt @R=20pt
{
~ e_iVR(t)~   \ar[r]^-{e_i(\ph t)}   \ar[d]_-{\V_i(Rt)}    &
~e_iSU(t)~   \ar[d]^-{\uS_i(Ut)}   &
~VRx +_i VRy~   \ar[r]^-{\ph x +_i \ph y}   \ar[d]_-{\V_i(Rx, Ry)}    &
~SUx +_i SUy~   \ar[d]^-{\uS_i(Ux, Uy)}
\\ 
~V(e_iRt)~   \ar[d]_-{V\R_i(t)}    &   
~S(e_iUt)~   \ar[d]^-{S\U_i(t)}    &
~V(Rx +_i Ry)~   \ar[d]_-{V\R_i(x, y)}    &   
~S(Ux +_i Uy)~   \ar[d]^-{S\U_i(x, y)}
\\ 
~VR(e_it)~   \ar[r]_-{\ph (e_it)}    &   ~SU(e_it)~   &
~VR(z)~   \ar[r]_-{\ph (z)}    &   ~SU(z)~
}
    \label{2.2.4} \end{array} \end{equation}

\Ndt  (coh.b) If $ p \le i < q $ (so that $ R, S $ are lax and $ U, V $ are colax in direction $ i$), 
we have commutative diagrams:
    \begin{equation} \begin{array}{c} 
    \xymatrix  @C=28pt @R=20pt
{
~ e_iVR(t)~   \ar[r]^-{e_i(\ph t)}    &
~e_iSU(t)~   \ar[d]^-{\uS_i(Ut)}   &
~VRx +_i VRy~   \ar[r]^-{\ph x +_i \ph y}      &
~SUx +_i SUy~   \ar[d]^-{\uS_i(Ux, Uy)}
\\ 
~V(e_iRt)~   \ar[d]_-{V\R_i(t)}   \ar[u]^-{\V_i(Rt)}    &   
~S(e_iUt)~     &
~V(Rx +_i Ry)~   \ar[d]_-{V\R_i(x, y)}   \ar[u]^-{\V_i(Rx, Ry)}  &   
~S(Ux +_i Uy)~ 
\\ 
~VR(e_it)~   \ar[r]_-{\ph (e_it)}    &   ~SU(e_it)~    \ar[u]_-{S\U_i(t)} &
~VR(z)~   \ar[r]_-{\ph (z)}    &   ~SU(z)~    \ar[u]_-{S\U_i(x, y)}
}
    \label{2.2.5} \end{array} \end{equation}

\Ndt  (coh.c) If $ i < p < q $ (so that $ R, S, U, V $ are colax in direction $ i$), we have commutative diagrams as in \eqref{2.2.4}, with all vertical arrows reversed.

	The $p$- and $q$-composition of these cubes are both defined using componentwise 
the transversal composition of a chiral multiple category. Namely, for a consistent 
matrix of $pq$-cubes and $x \in \iA$
%
    \begin{equation}
    \begin{array}{c}  \xymatrix  @C=12pt @R=6pt  
{
~\bu~  \arRb^R   \arDb_{U\,}  &&  
~\bu~   \arRb^{R'}  \arDb^{\, V}  &&  ~\bu~  \arDb^{\, W}
\\
&  \ph  &&  \psi   &&&&   \bu   \ar[r]^{p}   \ar[d]^{q}  &
\\
~\bu~  \ar[rr]|{~S~}  \arDb_{U'\,}  &&  
~\bu~   \ar[rr]|{~S'~} \arDb^{\, V'}  &&  ~\bu~ \arDb^{\, W'}  &&&
\\
&  \si  &&  \tau  &&&&&
\\
~\bu~ \arRb_T     &&  ~\bu~ \arRb_{T'}     &&  ~\bu~   
}
    \label{2.2.6} \end{array} \end{equation}
    \begin{equation} \begin{array}{lr}
(\ph  +_p \psi )(x)  =  \psi (Rx) +_0 S'(\ph x)\c  WR'Rx \to S'VRx \to S'SUx,
\\[5pt]
(\ph  +_q \si)(x)  =  V'(\ph x) +_0 \si(Ux)\c  V'VRx \to V'SUx \to TU'Ux.
    \label{2.2.7} \end{array} \end{equation}

\skp

	We prove below, in Theorem \ref{2.8}, that these composition laws are well-defined, 
i.e.\ the cells above do satisfy the previous coherence conditions. Moreover, they have been 
defined via the composition of transversal maps, and therefore are strictly unitary and associative.
	
	Finally, to verify the middle-four interchange law on the four double cells of diagram 
 \eqref{2.2.6}, we compute the composites $ (\ph  +_p \psi ) +_q (\si +_p \tau) $ and 
$ (\ph  +_q \si) +_p (\psi  +_q \tau) $ on an $\i$-cube $ x$, and we obtain the two transversal 
maps $ W'WR'Rx \to_0 T'TU'Ux $ of the upper or lower path in the following diagram
    \begin{equation} \begin{array}{c} 
    \xymatrix  @C=25pt @R=20pt
{
~W'WR'Rx~   \ar[r]^-{W'\psi Rx}    &
~W'S'VRx~   \ar[r]^-{W'S'\ph x}   \ar[d]_-{ \tau VRx}   \arv[rd]|-=    &
~W'S'SUx~   \ar[d]^-{\tau SUx} 
\\ 
&    ~T'V'VRx~   \ar[r]_-{T'V'\ph x}    &   ~T'V'SUx~   \ar[r]_-{T'\si Ux}    &   ~T'TU'Ux~
}
    \label{2.2.8} \end{array} \end{equation}

	The square commutes, by naturality of the double cell $ \tau  $ (with respect to 
the transversal map $ \ph x\c VR(x) \to_0 SU(x))$, so that the two composites coincide.

\subsection{Transversal maps of degree two}\label{2.3}
Given two $pq$-cubes
    \begin{equation}
\ph\c (U \; ^R_S  \; V),   \qq   \ph'\c (U' \; ^{R'}_{S'}  \; V')
    \label{2.3.1} \end{equation}
a transversal $pq$-map $ (F, G, F', G')\c \ph  \to_0 \ph' $ (of degree two and dimension three) 
is a quadruple of strict functors forming four transversal maps of degree 1
    \begin{equation} \begin{array}{ccc}
(F, G)\c R \to_0 R',   &~~&   (F', G')\c S \to_0 S',
\\[5pt]
 (F, F')\c U \to_0 U',   &&   (G, G')\c V \to_0 V',
    \label{2.3.2} \end{array} \end{equation}
%
%
    \begin{equation*}
    \begin{array}{c}  \xymatrix  @C=8pt   @R=15pt
{
~\iA~   \ar[rrr]^-R    \ar[dd]_U   \ar[rd]^>>>>F  &&&
~\bu~     \ar[rd]^-G   &&&&
~\iA~    \ar[rrr]^R   \ar[dd]_-U   &  \ar@{-->}@/^/[ld]^-\ph   &&
~\bu~    \ar[dd]_-{V}   \ar[rd]^-G
\\ 
& ~\bu~    \ar[rrr]^-{R'}   \ar[dd]_-{U'} \arv[rru]|-=  &   \ar@{-->}@/^/[ld]^-{\ph'}  &&   
~\bu~  \ar[dd]^-{V'}  &&&&&&&   ~\bu~  \ar[dd]^-{V'}   &&&
~\bu~  \ar[rr]^-p   \ar[rd]^>>>>0  \ar[d]_<<<<{q\,}  &&
\\ 
~\bu~  \ar[rd]_{F'}  \arv[ru]|-=  &&&&&&&
~\bu~ \ar[rrr]^-{R'}    \ar[rd]_{F'}    \arv[rrrrd]|-=  &&&
~\bu~  \ar[rd]_<{G'}  \arv[ru]|-=   &&&&&
\\ 
&   ~\bu~    \ar[rrr]_-{S'}   &&&  ~\iB~   &&&
&   ~\bu~    \ar[rrr]_-{S'}    &&&  ~\iB~ 
}
    \label{2.3.2bis} \end{array} \end{equation*}
and such that `the cube commutes', in the sense that, for every $\i$-cube $ x $ of $ \iA$, 
the following transversal maps of $\iB$ coincide
    \begin{equation}
G'(\ph x)\c G'VR(x) \to G'SU(x),   \qq   \ph'(Fx)\c VR'F(x) \to S'U'F(x).
    \label{2.3.3} \end{equation}

	We have thus defined the triple category $ \trp_{0pq}(\iCmc) $ of chiral multiple 
categories, with strict functors and $p$- and $q$-morphisms (for $ 0 < p < q \le \infty $). Its 
indices vary in the pointed ordered set $ \{0, p, q\}$.

\subsection{Three-dimensional cubes}\label{2.4}
A $pqr$-{\em cube} (for $ 0 < p < q < r \le \infty $) will be a `commutative cube' $ \Pi$, 
coskeletally determined by its six faces:

\Ndt  - two $pq$-cubes $ \ph, \psi  \, $ (the faces $ \dda_r \Pi$),

\Ndt  - two $pr$-cubes $ \pi, \rho  \, $ (the faces $ \dda_q \Pi$),

\Ndt  - two $qr$-cubes $ \om, \ze  \, $ (the faces $ \dda_p \Pi$),
%
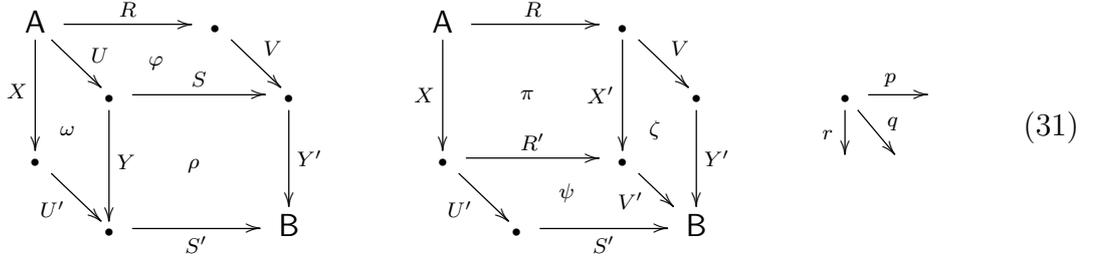
    \begin{equation}
    \begin{array}{c}  \xymatrix  @C=8pt   @R=15pt
{
~\iA~   \ar[rrr]^-R    \ar[dd]_X   \ar[rd]^>>>>U  &&&
~\bu~     \ar[rd]^-V   &&&&
~\iA~    \ar[rrr]^R   \ar[dd]_-X  &&&
~\bu~    \ar[dd]_-{X'}   \ar[rd]^-V
\\ 
& ~\bu~    \ar[rrr]^-S   \ar[dd]^Y \arv[rru]|-\ph  &&&   
~\bu~  \ar[dd]^-{Y'}  &&&&&&&   ~\bu~  \ar[dd]^-{Y'}   &&&
~\bu~  \ar[rr]^p   \ar[rd]^>>>>q  \ar[d]_<<<<{r\,}  &&
\\ 
~\bu~  \ar[rd]_{U'}  \arv[ru]|-\om  &&&&&&&
~\bu~ \ar[rrr]^-{R'}    \ar[rd]_{U'}    \arv[rrruu]|-\pi  \arv[rrrrd]|-\psi  &&&
~\bu~  \ar[rd]_<<<{V'}  \arv[ru]|-\ze   &&&&&
\\ 
&   ~\bu~    \ar[rrr]_-{S'}  \arv[rrruu]|-\rho  &&&  ~\iB~   &&&
&   ~\bu~    \ar[rrr]_-{S'}    &&&  ~\iB~ 
}
    \label{2.4.1} \end{array} \end{equation}

	The commutativity condition means that, for every $\i$-cube $ x $ of $ \iA$, 
the following composed transversal arrows in $ \iB$ coincide  
    \begin{equation} \begin{array}{c}
 S'\om x.\rho Ux.Y'\ph x \c  Y'VR(x)   \to   Y'SU(x)   \to   S'YU(x)   \to   S'U'X(x)
\\[5pt]
\psi Xx.V'\pi x.\ze Rx:  Y'VR(x) = V'X'R(x) = V'R'X(x) = S'U'X(x).
    \label{2.4.2} \end{array} \end{equation}

	These cubes are composed in direction $ p, q$, or $ r$, by pasting cubes (with the 
operations of 2-dimensional cubes). Again, these operations are associative, unitary and 
satisfy the middle-four interchange by pairs.

\subsection{Higher items}\label{2.5}
A transversal $pqr$-map $ F\c \Pi \to_0 \Pi' $ between $pqr$-cubes is determined by its 
boundary, a face-consistent family of six transversal maps of degree two (and dimension 
three)
    \begin{equation}
\dda_j F\c  \dda_j \Pi \to_0 \dda_j \Pi'   \qq    (\al = \pm, \; j \in \{p, q, r\}),
    \label{2.5.1} \end{equation}
under no other conditions. Their operations are computed on such faces.

	We have thus defined a quadruple category of chiral multiple categories, with strict functors 
and $p$-, $q$-, $r$-morphisms (for extended positive integers $ p < q < r$). The indices vary in 
the pointed ordered set $ \{0, p, q, r\}$.

	Finally, we have the multiple category $ \iCmc $ (indexed by the ordinal 
$ \bom + \one$), where each cell of dimension $ \ge 4 $ (starting with the transversal 
maps of degree 3 considered above and the cubes of dimension 4, not yet considered) is 
determined by a face-consistent family of all its iterated faces of dimension 3.

	In the truncated case we have the $(n+1)$-dimensional multiple category $ \iCmc_\n $ of (small) chiral 
$\n$-multiple categories, where the objects are indexed by the ordinal $\n = \{0,..., n-1\}$, while 
$ \iCmc_\n $ is indexed by $\n+\one $ (the previous $ \infty  $ being replaced by $ n$). But 
one should note that $ \iCmc_\n $ is {\em not} an ordinary truncation of $ \iCmc$,  as its 
objects too are truncated.

	$\iCmc $ is a substructure of the - similarly defined - multiple category $ \iInc $ of 
small infinite dimensional intercategories, and $ \iCmc_\n $ is a substructure of the 
$(n+1)$-dimensional multiple category $ \iInc_\n $ of small $\n$-intercategories.

\subsection{Comments}\label{2.6} 
These multiple categories are related to various double or triple categories previously 
constructed.

\Ndt (a) A chiral 1-multiple category is just a category, and $ \iCmc_\one $ is the double 
category of small categories, with commutative squares of functors as double cells.

\Ndt (b) A chiral 2-multiple category is a weak double category. Let us recall that we have already 
studied (in \cite{GP2}, Section 2) the double category $ \Dbl $ of weak double categories, with 
lax and colax functors - where double adjunctions live. Later $ \Dbl $ has been extended to a 
triple category $ \iS\Dbl $ of weak double categories, with strict, lax and colax functors 
\cite{GP8}, Section 1); in the latter all 2-dimensional cells are inhabited by possibly non-trivial 
transformations, while in $ \iCmc_2 $ the 01- and 02-cells are `commutative squares', 
inhabited by identities. Thus $ \iCmc_\two$ extends $ \Dbl $ but is a triple subcategory of 
$ \iS\Dbl$.

\Ndt (c) Multiple adjunctions live in the double category $ \Cmc $ of chiral multiple categories, 
with lax and colax multiple functors (\cite{GP10}, Section 2). This can be extended to a triple 
category $ \iS\Cmc $ of chiral multiple categories, with strict, lax and colax functors, where 
again all 2-dimensional cells are inhabited by possibly non-trivial transformations. Then 
$ \iS\Cmc $ contains the triple category obtained from $ \iCmc $ by restricting to the 
multi-indices $ \i \sub \{0, 1, \infty \}$. But there seems to be no reasonable way of 
extending $ \iCmc $ with non-trivial quintets on $0p$-cells, for $ 1 < p < \infty $.

\Ndt (d) The quadruple category $ \iInc_\thr $ of 3-dimensional intercategories is an extension 
of the triple category $ \iICat $ of \cite{GP6}, Section 6, obtained by adding strict functors in 
the transversal direction and `commutative transversal cells'.

\subsection{The quintet type}\label{2.7}
There is a multiple functor
    \begin{equation}
\tv\c \iCmc \to \iQ(\Mlt(\bCat)),
    \label{2.7.1} \end{equation}
with values in the multiple category of quintets over the 2-category of multiple objects in 
$ \bCat $ (indexed by $ \N^*$), which sends:

\Ndt  - a chiral multiple category $ \iA $ to the multiple object $ \tv(\iA) = (\tv_\i \iA)_{\i  \sub \N^*}$ 
of its transversal categories,

\ndt  - a $p$-morphism $ F\c \iA \to_p \iB $ to the multiple morphism 
$ \tv(F) = (F_\i \c \tv_\i \iA \to \tv_\i \iB)_\i  $ formed by its transversal components, the 
ordinary functors $ F_\i  $ (for $ p \ge 0$),

\ndt  - a $pq$-cell $ \ph \c VR \todash SU $ (as in \eqref{2.2.1} or also in \eqref{2.1.6}) to the 
multiple transformation $ \tv(\ph) = (\ph_\i \c V_\i R_\i  \to S_\i U_\i )_\i  $ formed by its 
transversal components, the ordinary natural transformations $ \ph_\i  $ (for $ 0 \le p < q$),

\ndt  - each higher cell $ \Phi $ of $ \iCmc $ to the corresponding cell of the codomain, by acting on 
all the 2-dimensional faces of $ \Phi$.

\skp

	This multiple functor makes $ \iCmc $ into a multiple category of quintet type, in the sense 
of Section 1. Moreover $ \iCmc $ has coskeletal dimension 3: note that the 3-dimensional 
cubes of \eqref{2.4.1} are determined by their boundary, {\em under} a commutativity 
condition which is not automatically satisfied.

\begin{thm}\label{2.8}
In $ \iCmc $ the composition law $ \ph +_p \psi  $ of $pq$-cubes is well-defined by the 
formula \eqref{2.2.8}
    \begin{equation}
(\ph  +_p \psi )(x)  =  \psi (Rx) +_0 S'(\ph x)\c  WR'Rx \to S'VRx \to S'SUx,
    \label{2.8.1} \end{equation}
in the sense that this family of transversal maps does satisfy the conditions 
(coh.a) - (coh.c) of \ref{2.2}.
\end{thm}
\begin{proof}
The argument is an extension of a similar one for the double category $ \Dbl $ in 
\cite{GP2}, Section 2, or for the double category $ \Cmc $ in \cite{GP10}, Section 2, 
taking into account the mixed laxity of the present `functors'. We prove the three coherence 
axioms with respect to a composed cube $ z = x +_i y $ in $ A_\i $; one would work in a 
similar way for a degenerate cube $e_i(t)$, with $ t \in A_{\i|i}$.

	First we prove (coh.a), letting $ p < q \le i$, so that all our functors $ R, R', S, S', U, V, W $ 
are lax in direction $ i$. This amounts to the commutativity of the outer diagram 
below, formed of transversal maps (the index $ i $ being omitted in $ +_i$ and in all 
comparisons $ \R_i, \R'_i$, etc.)
%
    \begin{equation*} \begin{array}{c} 
    \xymatrix  @C=50pt @R=20pt
{
WR'Rz~  \ar[r]^-{\psi Rz}  & ~S'VRz~  \ar[r]^-{S'\ph z}  & ~S'SUz 
\\ 
WR'(Rx + Ry)~   \ar[r]^-{\psi (Rx + Ry)}  \ar[u]^-{WR'\R}  &
~S'V(Rx + Ry)   \ar[u]_-{S'V\R}   &  S'S(Ux + Uy)   \ar[u]_-{S'S\U}
\\ 
W(R'Rx + R'Ry)   \ar[u]^-{W\R'R}   &  
S'(VRx + VRy)~  \ar[r]^-{S'(\ph x + \ph y)}   \ar[u]_-{S'\V R}  &  ~S'(SUx + SUy)  \ar[u]_-{S'\uS U}
\\ 
WR'Rx + WR'Ry~   \ar[r]_-{\psi Rx + \psi Ry}  \ar[u]_-{\W R'R}   &
~S'VRx + S'VRy~  \ar[r]_-{S'\ph x + S'\ph y}  \ar[u]_-{\uS'VR}  &
~S'SUx + S'SUy  \ar[u]_-{\uS' SU}
}
    \label{2.8.2} \end{array} \end{equation*}

	Indeed, the two hexagons commute by (coh.a), applied to $ \ph  $ and $ \psi $, 
respectively. The upper rectangle commutes by naturality of $ \psi  $ on $ \R_i(x, y)$. 
The lower rectangle commutes by axiom (lmf.2) \cite{GP8}, Section 3.9, on the lax 
functor $ S'$, with respect to the transversal $\i$-maps $ \ph x\c VR(x) \to_0 SU(x) $ and 
$ \ph y\c VR(y) \to_0 SU(y)$
    \begin{equation}
	S'(\ph x +_i \ph y).\uS'_i(VR(x), VR(y))  =  \uS'_i(SU(x), SU(y)).(S'(\ph x) +_i S'(\ph y)).
    \label{2.8.3} \end{equation}

	The proof of (coh.c) is transversally dual to the previous one. 
	To prove (coh.b) we let $ p \le i < q$, so that $ R, R', S, S' $ are lax, while $ U, V, W $ are colax in direction $ i$. We reverse the comparisons $ \U_i, \V_i, \W_i $ in the diagram above 
%
    \begin{equation*} \begin{array}{c} 
    \xymatrix  @C=50pt @R=20pt
{
WR'Rz~  \ar[r]^-{\psi Rz}  & ~S'VRz~  \ar[r]^-{S'\ph z}  & ~S'SUz  \ar[d]^-{S'S\U} 
\\ 
WR'(Rx + Ry)~   \ar[r]^-{\psi (Rx + Ry)}  \ar[u]^-{WR'\R}  &
~S'V(Rx + Ry)   \ar[u]_-{S'V\R}     \ar[d]^-{S'\V R}  &  S'S(Ux + Uy)
\\ 
W(R'Rx + R'Ry)   \ar[u]^-{W\R'R}  \ar[d]_-{\W R'R}   &  
S'(VRx + VRy)~  \ar[r]^-{S'(\ph x + \ph y)}  &  ~S'(SUx + SUy)  \ar[u]_-{S'\uS U}
\\ 
WR'Rx + WR'Ry~   \ar[r]_-{\psi Rx + \psi Ry}   &
~S'VRx + S'VRy~  \ar[r]_-{S'\ph x + S'\ph y}  \ar[u]_-{\uS'VR}  &
~S'SUx + S'SUy  \ar[u]_-{\uS' SU}
}
    \label{2.8.4} \end{array} \end{equation*}
and note that the two hexagons commute, by (coh.b) on $ \ph  $ and $ \psi $, 
while the rectangles are unchanged.

\end{proof}

	
\section{Generalised quintets for chains of adjunctions}\label{s3}
	We now construct a non-cubical multiple category of generalised quintets, out of a 
sequence of 2-categories and 2-functors. Then we deduce a non-cubical multiple category 
$ \iAdj(\bC) $ where an $i$-directed arrow is a {\em chain of adjunctions} 
$ u_0 \adj u_1 \adj ... \adj u_i $ of length $ i$, in a fixed 2-category $ \bC$.

	Examples of unbounded chains of adjunctions can be found in \cite{Bo}.

\subsection{A multiple category derived from a sequence of 2-categories}\label{3.1}
We start from a sequence $ \cC $ of 2-functors of 2-categories, {\em which are supposed to be 
the identity on a common set of objects} $ S$
    \begin{equation} \begin{array}{c} 
    \xymatrix  @C=20pt @R=20pt
{
(\cC)   &&&&   ~...~   \ar[r]     &   ~\bC_i~   \ar[r]^-{U_i}     &   ~~...~~  \ar[r]  &
~\bC_2~    \ar[r]^-{U_2} &   ~\bC_1~    \ar[r]^-{U_1} &   ~\bC_0~
}
    \label{3.1.1} \end{array} \end{equation}
    \begin{equation*}
U_{ji}  =  U_{i+1} ... U_j\c \bC_j \to \bC_i    \qqq    (0 \le  i < j),
    \label{3.1.1bis} \end{equation*}

\Ndt  and we construct a multiple category $ \iM = \iGQ(\cC ) $ of {\em generalised quintets} over 
the sequence $ \cC$, or {\em quintets modulo} $ U$. The 2-functors $ U_{ji} $ are called the 
{\em forgetful functors} of $ \cC $.

\Ndt  (a) The set of objects of $ \iM $ is $ S$. In dimension 1, an $i$-cell $ u\c X \to_i Y $ is a 
$\bC_i$-morphism $ (i \ge 0)$. They form the category $\cat_i(\iM) $ underlying $\bC_i$.

\Ndt  (b) In dimension 2 an $ij$-cell (for $ 0 \le i < j$) is a $U$-{\em quintet}, or 
{\em quintet modulo} $ U_{ji}\c \bC_j \to \bC_i$. It amounts to: two $\bC_i$-morphisms 
$ r, s $ (its $j$-faces), two $\bC_j$-morphisms $ u, v $ (its $i$-faces) and a 2-cell 
$ \ph  $ in $ \bC_i $
%
    \begin{equation} \begin{array}{c} 
    \xymatrix  @C=12pt @R=5pt
{
~A~   \ar[rr]^-{r}   \ar[dd]_-{u}   &&   ~B~   \ar[dd]^-{v}    &&&&
\bu   \ar[r]^i   \ar[d]^j  &
\\ 
& \ph  &&&  \ph \c (U_{ji}v)r \to s(U_{ji}u)\c A \to D.  &&&
\\ 
~C~   \ar[rr]_-{s}    &&   ~D~
}
    \label{3.1.2} \end{array} \end{equation}

	Such an $ij$-cell has an obvious underlying cell $ |\ph | $ in the double category 
$ \Q(\bC_i) $ of quintets over $ \bC_i$. $U$-quintets inherit from the latter two composition laws 
in directions $ i $ and $ j$, and form a double category $ \dbl_{ij}(\iM) $ with a cellwise-faithful 
double functor $ \dbl_{ij}(\iM) \to \Q(\bC_i)$.
	
\Ndt  (c) In dimension 3 an $ijk$-cell (for $ 0 \le i < j < k$) is a cube of $U$-quintets whose image 
in $ \bC_i $ commutes
%
    \begin{equation}
    \begin{array}{c}  \xymatrix  @C=8pt   @R=15pt
{
~\bu~   \ar[rrr]^-r    \ar[dd]_x   \ar[rd]^>>>>u  &&&
~\bu~     \ar[rd]^-v   &&&&
~\bu~    \ar[rrr]^r   \ar[dd]_-x  &&&
~\bu~    \ar[dd]_-{x'}   \ar[rd]^-v
\\ 
& ~\bu~    \ar[rrr]^-s   \ar[dd]^y \arv[rru]|-\ph  &&&   
~\bu~  \ar[dd]^-{y'}  &&&&&&&   ~\bu~  \ar[dd]^-{y'}   &&&
~\bu~  \ar[rr]^>>>>i   \ar[rd]^>>>>j  \ar[d]_<<<<{k\,}  &&
\\ 
~\bu~  \ar[rd]_{u'}  \arv[ru]|-\om  &&&&&&&
~\bu~ \ar[rrr]^-{r'}    \ar[rd]_{u'}    \arv[rrruu]|-\pi  \arv[rrrrd]|-\psi  &&&
~\bu~  \ar[rd]_<{v'}  \arv[ru]|-\ze   &&&&&
\\ 
&   ~\bu~    \ar[rrr]_-{s'}  \arv[rrruu]|-\rho  &&&  ~\bu~   &&&
&   ~\bu~    \ar[rrr]_-{s'}    &&&  ~\bu~ 
}
    \label{3.1.3} \end{array} \end{equation}

   More precisely, we have six $U$-quintets determined by 2-cells in $ \bC_i $ or $ \bC_j $
    \begin{equation*} \begin{array}{lllr}
\text{- $ij$-cells:}    &    \ph\c (U_{ji}v)r \to s(U_{ji}u),  &
\psi \c (U_{ji}v')r' \to s'(U_{ji}u')    &    \text{(2-cells in }  \bC_i),
\\[5pt]
\text{- $ik$-cells:}   &    \pi \c (U_{ki}x')r \to r'(U_{ki}x),    &    
\rho \c (U_{ki}y')s \to s'(U_{ki}y)    &    \text{(2-cells in }  \bC_i),
\\[5pt]
\text{- $jk$-cells:}    &    \om \c (U_{kj}y)u \to u'(U_{kj}x),    &  
\ze \c (U_{kj}y')v \to v'(U_{kj}x')    &    \text{(2-cells in }  \bC_j).
    \label{3.1.4} \end{array} \end{equation*}

	They must satisfy the following commutativity relation, in $ \bC_i $
    \begin{equation*}
Uy'.\ph \te \rho .Uu \te s'.U\om  =  U\ze .r \te Uv'.\pi  \te \psi .Ux\c  Uy'.Uv.r \to s'.Uu'.Ux,
    \label{3.1.5} \end{equation*}
where $ U $ stands for $ U_{ji} $ or $ U_{ki} $ when it operates on 
$ \bC_j $ or $ \bC_k$, respectively.

	The compositions in directions $ i, j, k $ amount to compositions of faces in the double 
categories $ \dbl_{ij}(\iM)$, $\dbl_{jk}(\iM) $ and $ \dbl_{ik}(\iM)$. We have thus a triple 
category $ \trp_{ijk}(\iM)$.

\Ndt  (d) Every cell of dimension $ n > 3 $ is an $n$-dimensional cube whose 2-dimensional 
faces are $U$-quintets, under the condition that each 3-dimensional face in direction $ ijk $ 
be an $ijk$-cell, as defined above.

\subsection{Comments}\label{3.2}
As a particular case, if we start from a 2-category $ \bC $ and build the sequence $ \cC $ 
of its identities, this procedure gives the multiple category $ \iQ(\bC) $ {\em of quintets over} 
$ \bC$.

	In general $ \iGQ(\cC) $ is a multiple category of generalised quintets, with respect to 
the obvious forgetful multiple functor $ U\c \iGQ(\cC) \to$ $\iQ(\bC_0)$; its coskeletal dimension 
is at most 3.

	If we assume that all the forgetful functors $ U_i $ are faithful, this multiple functor 
satisfies condition \ref{1.4}(i), and $ \iGQ(\cC) $ becomes a multiple category of quintet type, 
over the 2-category $\bC_0$.

\subsection{Chains of adjunctions}\label{3.3}
The previous construction allows us to build a multiple category $ \iAdj(\bC) $ of 
{\em chains of adjunctions} over a 2-category $ \bC$. 

	We begin by forming the 2-category $ \bC_i = \rAdj_i(\bC) $ of $i$-chains of adjunctions 
in $ \bC$, for $ i \ge 0$.

	For $ i = 0 $ we just let $ \rAdj_0(\bC) = \bC$. In general, an object of $ \rAdj_i(\bC) $ 
is an object of $ \bC$, and a morphism 
    \begin{equation}
u  =  (u_0,..., u_i; \eta _1, \ep_1,..., \eta _i, \ep_i)\c X \to_i Y,
    \label{3.3.1} \end{equation}
is a chain of adjunctions $ u_0 \adj u_1 \adj ... \adj u_i $ in  \bC,  with   
    \begin{equation} \begin{array}{l}
u_0\c X \to Y,   \q   u_1\c Y \to X,   \q   u_2\c X \to Y,...
\\[5pt]
(\eta _1\c 1_X \to u_1u_0, \;  \ep_1\c u_0u_1 \to 1_Y)\c u_0 \adj u_1  ~~~
(u_0\eta _1  \te  \ep_1u_0 = 1, \;  \eta _1u_1 \te  u_1\ep_1 = 1),
\\[5pt]
(\eta _2\c 1_Y \to u_2u_1, \;   \ep_2\c u_1u_2 \to 1_X)\c u_1 \adj u_2  ~~~
(u_1\eta _2  \te  \ep_2u_1 = 1, \;  \eta_2u_2  \te  u_2\ep_2 = 1),
\\[5pt]
...
    \label{3.3.2} \end{array} \end{equation}

	The chain will also be written as $ u = (u_0, u_1, u_2,...)$, leaving units and 
counits understood.

	The composition of $ u $ with a consecutive arrow $ v\c Y \to_i Z $ is just an 
iterated composition of adjunctions
    \begin{equation}
vu  =  (v_0, v_1, v_2,...).(u_0, u_1, u_2,...)  =  (v_0u_0, u_1v_1, v_2u_2,...)\c  X \to_i Z,
    \label{3.3.3} \end{equation}
that ends with $ v_iu_i $ (resp. $ u_iv_i$) when $ i $ is even (resp. odd).  The identity of $X$ is 
a chain of identities.

A 2-cell of $ \rAdj_i(\bC)$
    \begin{equation}
\ph  =  (\ph_0,..., \ph_i)\c  (u_0,..., u_i; \eta_1, \ep_1,... \,) \to (v_0,..., v_i; \eta_1, \ep_1,... \,)\c  X \to_i Y,
    \label{3.3.4} \end{equation}
    \begin{equation*} \begin{array}{c} 
    \xymatrix  @C=20pt @R=20pt
{
~u_0~  \ar@{-|}[r]   \ar[d]_-{\ph_0}    &   ~u_1~  \ar@{-|}[r]    &
~u_2~  \ar@{-|}[r]   \ar[d]^-{\ph_2}    &...
\\ 
~v_0~  \ar@{-|}[r]    &   ~v_1~  \ar@{-|}[r]   \ar[u]_-{\ph_1}    &
~v_2~  \ar@{-|}[r]    &...
}
    \label{3.3.4bis} \end{array} \end{equation*}
amounts to a 2-cell $ \ph_0\c u_0 \to v_0 $ of $ \bC$, together with its mates in the 
corresponding adjunctions
    \begin{equation} \begin{array}{lr}
\ph_0\c u_0 \to v_0\c X \to Y,
\\[5pt]
\ph_1\c v_1 \to u_1\c Y \to X    \q  &   (\ph_1 = \eta_1v_1 \te u_1\ph_0v_1 \te u_1\ep_1),
\\[5pt]
\ph_2\c u_2 \to v_2\c X \to Y    &   (\ph_2 = \eta_2u_2 \te v_2\ph_1u_2 \te v_2\ep_2),
\\[5pt]
...
    \label{3.3.5} \end{array} \end{equation}

	The vertical composition of $ \ph $ with a 2-cell 
$ \psi  = (\psi_0,..., \psi_i)\c v \to w\c X \to_i Y $ amounts to the composite $ \ph_0 \te \psi_0$, 
and can be written as
    \begin{equation}
\ph \te \psi  =  (\ph_0 \te \psi_0, \psi_1 \te \ph_1, \ph_2 \te \psi_2,... )\c u \to w\c X \to_i Y,
    \label{3.3.6} \end{equation}
since mates agree with composition (in a contravariant way, of course). The identity of $u$ 
is the sequence $(\id(u_0),..., \id(u_i))$.

	The whisker composition of a 2-cell
 $ \ph = (\ph_0, \ph_1, \ph_2,...)\c u \to v\c X \to_i Y $ with arrows $ r\c X' \to_i X $ and 
 $ s\c Y \to_i Y' $ is
    \begin{equation}
s\ph r  =  (s_0\ph _0r_0,  r_1\ph _1s_1,  s_2\ph _2r_2,...)\c  sur \to svr\c  X' \to_i Y',
    \label{3.3.7} \end{equation}
since mates agree with whisker composition (again in a contravariant way).

	The 2-functor $ U_{ji}\c \bC_j \to \bC_i $ forgets the components of arrows and cells of 
all indices $ i+1,..., j$.

	Finally we define $\iAdj(\bC) $ as the multiple category $ \iGQ(\cC) $ of this sequence 
of 2-categories.

\subsection{Truncation}\label{3.4}
The multiple category $ \iAdj(\bC) $ extends the double category $ \Adj(\bC) $ of morphisms 
and adjunctions of $ \bC $ considered in \cite{GP1}, Section 3.5 (written down for 
$ \bC = \bCat$, but the general case works in the same way).

	To show that $\Adj(\bC) $ can be identified with the 2-dimensional truncation of 
$ \iAdj(\bC)$, we begin by noting that the 0- and 1-cells of $ \iAdj(\bC) $ coincide with the 
horizontal and vertical arrows of $ \Adj(\bC)$; we are left with examining the 2-dimensional 
cells.

	In a 01-cell of $ \iAdj(\bC) $ the vertical arrows are adjunctions 
$ u = (u_\bu, u^\bu , \eta, \ep)$  and  $v = (v_\bu, v^\bu , \eta', \ep') $ in $ \bC$ 

%
    \begin{equation} \begin{array}{c} 
    \xymatrix  @C=12pt @R=6pt
{
~A~   \ar[rr]^-{f}   \arDb_-{u}   &&   ~B~   \arDb^-{v}  &&&&
\bu   \ar[r]^{0}   \ar[d]^{1}  &
\\
&  \ph   &&&  \ph_\bu\c v_\bu f \to gu_\bu \c A \to D,  &&
\\ 
~C~  \ar[rr]_-{g}    &&   ~D~ 
}
    \label{3.4.1} \end{array} \end{equation}
and $ \ph $ amounts to a single 2-cell $ \ph_\bu  $ of $ \bC$, as above. But we can 
equivalently add a mate $ \ph^\bu \c fu^\bu  \to v^\bu g $
    \begin{equation} \begin{array}{c}
\ph^\bu =  \eta'fu^\bu  \te v^\bu \ph_\bu u^\bu  \te v^\bu g\ep\c \;  fu^\bu  \to v^\bu v_\bu fu^\bu  \to v^\bu gu_\bu u^\bu  \to v^\bu g,
    \label{3.4.2} \end{array} \end{equation}
and this completes the double cell $ \ph = (\ph_\bu , \ph^\bu ) $ as defined in \cite{GP1}.

\section{Generalised quintets for arrow bundles}\label{s4}
	We now study a similar construction of a non-cubical multiple category of generalised 
quintets, based on a sequence of 2-categories and 2-functors of a different shape.

	Then we deduce a non-cubical multiple category $ \iBnd(\bC) $ where an $i$-directed 
arrow is a `bundle' $ u = (u_0,..., u_i) $ of parallel arrows in a fixed 2-category $ \bC$.

\subsection{A second construction of generalised quintets}\label{4.1}
We start now from a sequence $ \cC $ of 2-functors which {\em raise the index}, and again 
are the identity on a common set $ S = \Ob \bC_i $ of objects
    \begin{equation} \begin{array}{c} 
    \xymatrix  @C=20pt @R=20pt
{
(\cC)   \qq  ~\bC_0~   \ar[r]^-{U_0}   &    ~\bC_1~   \ar[r]^-{U_1}     &   ~\bC_2~   \ar[r]     &   
~~...~~  \ar[r]  &  ~\bC_i~   \ar[r]^-{U_i}     &   ~~...~~
}
    \label{4.1.1} \end{array} \end{equation}
    \begin{equation*}
U_{ij}  =  U_{j-1} ... U_i\c \bC_i \to \bC_j    \qq    (0 \le  i < j),
    \label{4.1.1bis} \end{equation*}

\skp

	The 2-functors $ U_{ij} $ are called the {\em structural functors} of $ \cC$. We shall 
construct a multiple category $ \iM = \iGQ(\cC) $ of {\em generalised quintets} over the 
sequence $ \cC$, of coskeletal dimension 3. The construction is much the same as in 
\ref{3.1}, taking into account that the structural functors work the other way round.

\Ndt  (a) Again, the set of objects is $ S $ and an $i$-cell $ u\c X \to_i Y $ is a $\bC_i$-morphism 
$ (i \ge 0)$. They form the category $\cat_i(\iM) $ underlying $ \bC_i$.

\Ndt  (b) In dimension 2 an $ij$-cell (for $ 0 \le i < j$) is a $U$-{\em quintet}, or 
{\em quintet modulo} $ U_{ij}\c \bC_i \to \bC_j$. It amounts to: two $\bC_i$-morphisms 
$ r, s $ (its $j$-faces), two $\bC_j$-morphisms $ u, v $ (its $i$-faces) and a 2-cell 
$ \ph $ in $ \bC_j $
%
    \begin{equation} \begin{array}{c} 
    \xymatrix  @C=12pt @R=5pt
{
~A~   \ar[rr]^-{r}   \ar[dd]_-{u}   &&   ~B~   \ar[dd]^-{v}    &&&&
\bu   \ar[r]^i   \ar[d]^j  &
\\ 
& \ph  &&& \ph\c v.U_{ij}r \to U_{ij}s.u\c A \to D.  &&&
\\ 
~C~   \ar[rr]_-{s}    &&   ~D~
}
    \label{4.1.2} \end{array} \end{equation}

	Such an $ij$-cell has an obvious underlying cell $ |\ph| $ in the double category 
$ \Q(\bC_j) $ of quintets over $ \bC_j$. $U$-quintets inherit from the latter two composition 
laws in directions $ i $ and $ j$, and form a double category $\dbl_{ij}(\iM)$.

\Ndt  (c) In dimension 3 an $ijk$-cell (for $ 0 \le i < j < k$) is a cube of $U$-quintets whose image 
in $ \bC_k $ commutes
%
    \begin{equation}
    \begin{array}{c}  \xymatrix  @C=8pt   @R=15pt
{
~\bu~   \ar[rrr]^-r    \ar[dd]_x   \ar[rd]^>>>>u  &&&
~\bu~     \ar[rd]^-v   &&&&
~\bu~    \ar[rrr]^r   \ar[dd]_-x  &&&
~\bu~    \ar[dd]_-{x'}   \ar[rd]^-v
\\ 
& ~\bu~    \ar[rrr]^-s   \ar[dd]^y \arv[rru]|-\ph  &&&   
~\bu~  \ar[dd]^-{y'}  &&&&&&&   ~\bu~  \ar[dd]^-{y'}   &&&
~\bu~  \ar[rr]^>>>>i   \ar[rd]^>>>>j  \ar[d]_<<<<{k\,}  &&
\\ 
~\bu~  \ar[rd]_{u'}  \arv[ru]|-\om  &&&&&&&
~\bu~ \ar[rrr]^-{r'}    \ar[rd]_{u'}    \arv[rrruu]|-\pi  \arv[rrrrd]|-\psi  &&&
~\bu~  \ar[rd]_<{v'}  \arv[ru]|-\ze   &&&&&
\\ 
&   ~\bu~    \ar[rrr]_-{s'}  \arv[rrruu]|-\rho  &&&  ~\bu~   &&&
&   ~\bu~    \ar[rrr]_-{s'}    &&&  ~\bu~ 
}
    \label{4.1.3} \end{array} \end{equation}

  More precisely, we have six $U$-quintets determined by 2-dimensional cells in $ \bC_j $ 
or $ \bC_k$
    \begin{equation*} \begin{array}{lllr}
\text{- $ij$-cells:}    &   \ph\c v.U_{ij}r \to U_{ij}s.u,  &
\psi\c v'.U_{ij}r' \to U_{ij}s'.u'	    &    \text{(2-cells in }  \bC_j),
\\[5pt]
\text{- $ik$-cells:}   &    \pi \c x'.U_{ik}r \to U_{ik}r'.x,    &    
\rho \c r.U_{ik}s \to U_{ik}s'.y    &    \text{(2-cells in }  \bC_k),
\\[5pt]
\text{- $jk$-cells:}    &   \om\c y.U_{jk}u \to U_{jk}u'.x,    &  
\ze \c r.U_{jk}v \to U_{jk}v'.x'    &    \text{(2-cells in }  \bC_k).
    \label{4.1.4} \end{array} \end{equation*}

	They must satisfy the following commutativity relation, in the 2-category $ \bC_k $
$$	
r.U\ph \te \rho .Uu \te Us'.\om  =  \ze .Ur \te Uv'.\pi  \te U\psi.x\c\,  r.Uv.Ur  \to  Us'.Uu'.x,
$$	
where $ U $ stands for $ U_{ik} $ or $ U_{jk} $ when it operates on $ \bC_i $ or 
$ \bC_j$, respectively.

	The compositions in directions $ i, j, k $ amount to compositions of faces in the double 
categories $ \dbl_{ij}(\iM)$, $\dbl_{jk}(\iM) $ and $ \dbl_{ik}(\iM)$. We have thus a triple category $ \trp_{ijk}(\iM)$.

\Ndt  (d) Every cell of dimension $ n > 3 $ is an $n$-dimensional cube whose 2-dimensional 
faces are $U$-quintets, under the condition that each 3-dimensional face in direction 
$ ijk $ be an $ijk$-cell, as defined above.

\skp

	Again, the coskeletal dimension of the multiple category $ \iM = \iGQ(\cC) $ is at most 3. 
If all the functors $ U_i $ are embeddings of 2-categories (injectives on objects, morphisms 
and 2-cells), $ \iGQ(\cC) $ can be obtained as a multiple subcategory of $ \iQ(\bC_\infty )$, 
where $ \bC_\infty $ is the colimit of diagram \eqref{4.1.1} - or some larger structure if 
convenient. $ \iGQ(\cC) $ is thus a multiple category of quintet type over the 2-category 
$ \bC_\infty$.

\subsection{A multiple category of arrow bundles}\label{4.2}
As an example, we start from a 2-category $ \bC $ 
and construct a diagram $ \cC $ of 2-categories and 2-functors, of type \eqref{4.1.1}.

	For every $ i \ge 0$, $ \bC_i $ is a 2-category with the same objects of $\bC$, and 
{\em bundles} of arrows and cells
    \begin{equation} \begin{array}{ll}
u  =  (u_0,..., u_i)\c A \to_i B,   &   u_0,..., u_i  \in  \bC(A, B),
\\[5pt]
\ph  =  (\ph_0,..., \ph_i)\c u \to v\c A \to_i B,   \q   &
\ph_0\c u_0 \to v_0,..., \,  \ph_i\c u_i \to v_i   \; \inn  \;  \bC,
    \label{4.2.1} \end{array} \end{equation}
so that $ \bC_0 = \bC$. Furthermore the 2-functor $ U_i\c \bC_i \to \bC_{i+1} $ repeats the 
last arrow or cell of a bundle.

	We shall write $ U_{ij}(u_0,..., u_i) = (u_0,..., u_j)$, leaving as understood that 
$ u_i = ... = u_j$; similarly for 2-cells.

	Plainly we are just considering an increasing filtration of the 2-category $ \bC_\infty $ 
that has unbounded bundles $ (u_h)_{h\ge 0}$, $(\ph_h)_{h\ge 0}$ of parallel arrows and 
cells of $ \bC $ (eventually constant, if we want).
	
	Therefore the multiple category $ \iBnd(\bC) = \iGQ(\cC) $ {\em of arrow bundles over} 
$ \bC $ is a multiple subcategory of $ \iQ(\bC_\infty )$. It can be described as follows.
	
\Ndt  (a) The set of objects is $ \Ob \bC $ and an $i$-cell $ u = (u_0,..., u_i)\c A \to_i B $ is a 
bundle of $ i+1 $ parallel arrows of $ \bC$. They form the category $\cat_i(\iBnd(\bC)) = \bC_i$.

\Ndt  (b) In dimension 2 an $ij$-cell (for $ 0 \le i < j$) amounts to: two $\bC_i$-morphisms 
$ r, s $ (its  $j$-faces), two $\bC_j$-morphisms $ u, v $ (its $i$-faces) and a sequence 
$ \ph = (\ph_0,..., \ph_j) $ of 2-cells in $ \bC $
%
    \begin{equation} \begin{array}{c} 
    \xymatrix  @C=12pt @R=5pt
{
~A~   \ar[rr]^-{r}   \ar[dd]_-{u}    &   \ar@/^/[ld]^\ph   &   ~B~   \ar[dd]^-{v}    &&&&
\bu   \ar[r]^i   \ar[d]^j  &
\\ 
&&&& \ph_h\c v_hr_h \to s_hu_h  ~~~ (1 \le h \le j), &&&
\\ 
~C~   \ar[rr]_-{s}    &&   ~D~
}
    \label{4.2.2} \end{array} \end{equation}
where the sequences $ (r_h) $ and $ (s_h) $ are constant for $ h \ge i$.

\Ndt  (c) In dimension 3 an $ijk$-cell (for $ 0 \le i < j < k$) is a cubical diagram
%
    \begin{equation}
    \begin{array}{c}  \xymatrix  @C=8pt   @R=15pt
{
~\bu~   \ar[rrr]^-r    \ar[dd]_x   \ar[rd]^>>>>u  &&&
~\bu~     \ar[rd]^-v   &&&&
~\bu~    \ar[rrr]^r   \ar[dd]_-x  &&&
~\bu~    \ar[dd]_-{x'}   \ar[rd]^-v
\\ 
& ~\bu~    \ar[rrr]^-s   \ar[dd]^y \arv[rru]|-\ph  &&&   
~\bu~  \ar[dd]^-{y'}  &&&&&&&   ~\bu~  \ar[dd]^-{y'}   &&&
~\bu~  \ar[rr]^>>>>i   \ar[rd]^>>>>j  \ar[d]_<<<<{k\,}  &&
\\ 
~\bu~  \ar[rd]_{u'}  \arv[ru]|-\om  &&&&&&&
~\bu~ \ar[rrr]^-{r'}    \ar[rd]_{u'}    \arv[rrruu]|-\pi  \arv[rrrrd]|-\psi  &&&
~\bu~  \ar[rd]_<{v'}  \arv[ru]|-\ze   &&&&&
\\ 
&   ~\bu~    \ar[rrr]_-{s'}  \arv[rrruu]|-\rho  &&&  ~\bu~   &&&
&   ~\bu~    \ar[rrr]_-{s'}    &&&  ~\bu~ 
}
    \label{4.2.3} \end{array} \end{equation}

	Its six faces are determined by their boundary-arrows and by bundles of 2-cells of $ \bC$

    \begin{equation*} \begin{array}{lllr}
\text{- $ij$-cells:}  \q  &   \ph_h\c v_hr_h \to s_hu_h,  &
\psi_h\c v'_hr'_h \to s'_hu'_h	\q    &   (1 \le h \le j),
\\[5pt]
\text{- $ik$-cells:}   &   \pi _h\c x'_hr_h \to r'_hx_h,    &    
\rho _h\c y'_hs_h \to s'_h.y_h    &    (1 \le h \le k),
\\[5pt]
\text{- $jk$-cells:}    &  \om_h\c y_hu_h \to u'_hx_h,    &  
\ze _h\c y'_hv_h \to v'_hx'_h    &   (1 \le h \le k).
    \label{4.2.4} \end{array} \end{equation*}

	They must satisfy the following commutativity relations, in $ \bC$ 
$$
y'_h\ph_h \te \rho_hu_h \te s'_h\om_h  =  \ze _hr_h \te v'_h\pi _h \te \psi_hx_h\c  y'_hv_hr_h \to s'_hu'_hx_h	~~~  (1 \le h \le k).
$$

\ndt  (d) Every cell of dimension $ n > 3 $ is an $n$-cube whose 2-dimensional faces are 
$U$-quintets, under the condition that each 3-dimensional face in direction $ ijk $ be an 
$ijk$-cell, as defined above.

	
\section{The triple category of pseudo algebras}\label{s5}
	For a 2-monad $ T $ on a 2-category $ \bC $ we build a triple category of pseudo 
$T$-algebras, of quintet type. Then we show how algebras and normal pseudo algebras of 
graphs of categories are related to strict and weak double categories.

	Algebras and pseudo algebras for a 2-monad are studied in \cite{Bu, BlKP, Fi}.

\subsection{Reviewing pseudo algebras}\label{5.1}
We have a 2-monad $ T = (T, h, m) $ on the 2-category $ \bC$. A {\em pseudo algebra} in 
$ \bC $ is a quadruple $ (A, a, \om, \ka) $ consisting of an object $ A $ of $ \bC$, a map 
$ a\c TA \to A $ (the {\em structure}) and two vertically invertible cells (the {\em comparisons})
    \begin{equation} \begin{array}{lr}
\om\c 1_A \to a.hA   &    \text{(the {\em normaliser}),}
\\[5pt]
\ka\c a.Ta \to a.mA   \qq  &    \text{(the {\em extended associator}),}
    \label{5.1.1} \end{array} \end{equation}
    \begin{equation*} \begin{array}{c} 
    \xymatrix  @C=10pt @R=10pt
{
~A~   \ar[rr]^-{1}   \ar[dd]_-{hA}    &  \ar@/^/[ld]^\om  &   ~A~   \ar[dd]^-{1}   &&&&
~T^2A~   \ar[rr]^-{Ta}   \ar[dd]_-{mA}    &  \ar@/^/[ld]^\ka   &   ~TA~   \ar[dd]^-{a}
\\
&&&&&&&
\\ 
~TA~   \ar[rr]_-{a}  &&   ~A~   &&&&    ~TA~   \ar[rr]_-{a}    &&   ~A~
}
    \label{5.1.1bis} \end{array} \end{equation*}

	These data are to satisfy three conditions of coherence:
    \begin{equation} \begin{array}{c}
a.T\om \te \ka.ThA  =  1_a  =  \om a \te \ka.hTA,
\\[5pt]
\ka.T^2a \te \ka.mTA  =  a.T\ka \te \ka.TmA,
    \label{5.1.2} \end{array} \end{equation}
    \begin{equation*} \begin{array}{c} 
    \xymatrix  @C=5pt @R=20pt
{
~a~   \ar[r]^-{\om a}   \ar[d]_-{a.T\om}   \ar[rrdd]^-{1}    &
a.hA.a   \ar@{=}[r]  &   a.Ta.hTA   \ar[d]^-{\ka.hTA}  &
~a.Ta.T^2a~   \ar[rrr]^-{a.T\ka}   \ar[d]_-{\ka.T^2a}  &&&    a.Ta.TmA    \ar[d]^-{\ka.TmA}
\\ 
~a.Ta.ThA~    \ar[d]_-{\ka.ThA}    &&   ~a.mA.hTA~   \ar@{=}[d]  &
~a.mA.T^2a~    \ar@{=}[d]    &&&   ~a.mA.TmA~    \ar@{=}[d]
\\ 
~a.mA.ThA~    \ar@{=}[rr]    &&   ~a~   &
~a.Ta.mTA~     \ar[rrr]_-{\ka.mTA}    &&&   ~a.mA.mTA~
}
    \label{5.1.2bis} \end{array} \end{equation*}

	A (strict) {\em morphism of pseudo algebras} $ f\c (A, a, \om, \ka) \to (B, b, \om, \ka) $ 
is a morphism $ f\c A \to B $ of $ \bC $ which preserves the structure (note that the comparisons 
of pseudo algebras are always denoted by the same letters)
    \begin{equation}
b.Tf  =  f.a,	   \q   \om f  =  f\om ,   \q   \ka.T^2f  =  f\ka.
    \label{5.1.3} \end{equation}

	We write as $ \bPsa_2(T) $ the category of pseudo algebras and their strict morphisms.

\subsection{Normal pseudo algebras}\label{5.2}
We say that a pseudo algebra $ (A, a, \om , \ka) $ is {\em normal} if the normaliser $ \om $ 
is the identity, and therefore $ a.hA = 1_A$.

	Normal pseudo algebras are important, and should not be called `unitary', as the 
following example shows. We consider the 2-monad 
$ T\c \bCat \to \bCat$, where $ T(A) $ is the free strict monoidal category over the category $A$; 
its objects are the finite families $ (X_1,..., X_n)$ of objects of $ A$.

	A strict monoidal structure $ a\c TA \to A $ over a (small) category $ A $
    \begin{equation}
a(X_1,..., X_n)  =  \TE_i \, X_i,
    \label{5.2.1} \end{equation}
gives all finite tensor products in $ A$; the identity object $ I = a(\ul{e}) $ comes from the 
empty family $\ul{e}$.

	A {\em normal} pseudo algebra $ (A, a, \ka) $ amounts to an unbiased monoidal 
category, {\em with a trivial unary tensor} $ \TE(X) = X $ (of single objects of $ A$); 
the two unitors and the binary associator all come from the extended associator $ \ka $ 
(that operates on finite tensor products)
    \begin{equation}  \begin{array}{c}
    \xymatrix  @C=42pt @R=20pt  @H=30pt   @!0
{
\TE(\TE(\ul{e}), \TE(X))~  \ar[rr]  &&   ~\TE(X),   &~&
\TE(\TE(X), \TE(\ul{e}))~  \ar[rr]  && ~ \TE(X),
\\
 \TE(\TE(X), \TE(Y, Z))~     \ar[rrr]  &&&  ~ \TE(X, Y, Z)~   &&&
~\TE(\TE(X, Y), \TE(Z))   \ar[lll].
}
    \label{5.2.2}  \end{array}  \end{equation}

	In the general case each object $ X $ has an associated object $ \TE(X)$, isomorphic to 
$ X$, which - when the procedure is idempotent - can be viewed as a `normal form'.

	For instance, let $ A $ be the monoidal category of finite totally ordered sets, with 
$ X\te Y $ the ordinal sum (extending the sum $ X+Y $ by letting $ x < y $ for all $ x \in X$, 
$y \in Y$). Redefining $ \TE'(X_1,... , X_n) $ as the ordinal isomorphic to $ \TE(X_1,... , X_n) $ 
we have a pseudo algebra $ (A, a, \om, \ka) $ which is not normal, but has a trivial associator 
$ \ka = \id $.

\subsection{Lax, colax and pseudo morphisms}\label{5.3}
Let us come back to a 2-monad $(T, h, m) $ on the 2-category $ \bC $ and its pseudo 
algebras.

\Ndt  (a) A {\em lax morphism of pseudo algebras} $\f = (f, \ph)\c (A, a, \om, \ka) \to (B, b, \om, \ka)$ 
is a morphism $ f\c A \to B $ of $ \bC $ with a comparison cell $ \ph $ such that:
    \begin{equation} \begin{array}{c}
\ph\c b.Tf \to f.a,   \qq   \om f \te \ph.hA  =  f\om,
\\[5pt]
\ka.T^2f \te \ph.mA  =  b.T\ph \te \ph.Ta \te f\ka,
    \label{5.3.1} \end{array} \end{equation}
    \begin{equation*} \begin{array}{c} 
    \xymatrix  @C=22pt @R=20pt
{
~f~   \ar[r]^-{f\om}   \ar[d]_-{\om f}    &   ~fa.hA~  &
~b.Tb.T^2f~   \ar[r]^-{b.T\ph}   \ar[d]_-{\ka.T^2f}    &
~b.Tf.Ta~   \ar[r]^-{\ph.Ta}    &   ~fa.TA~  \ar[d]^-{f\ka}
\\ 
~b.hB.f~   \ar@{=}[r]    &   ~b.Tf.hA~    \ar[u]_-{\ph.hA}  &
~b.mB.T^2f~   \ar@{=}[r]    &   ~b.Tf.mA~   \ar[r]_-{\ph.mA}    &   ~fa.mA~
}
    \label{5.3.1bis} \end{array} \end{equation*}

	These morphisms compose: given $ (g, \ga)\c (B, b, \om , \ka) \to (C, c, \om, \ka)$, we let
    \begin{equation}
(g, \ga).(f, \ph)  =  (gf,  \ga.Tf \te g\ph),   \q   \ga.Tf \te g\ph\c  c.T(gf) \to gb.Tf \to gfa.
    \label{5.3.2} \end{equation}

	We have thus a category $ \Lx\bPsa_2(T)$, with identities 
$ \id (A, a, \om, \ka) = (\id A, 1_a)$. A {\em pseudo morphism} is a lax morphism $ (f, \ph) $ 
where the cell $ \ph $ is vertically invertible.

\Ndt (b) A {\em colax morphism of pseudo algebras} 
$\r = (r, \rho)\c (A, a, \om, \ka) \to (B, b, \om, \ka) $ 
is a morphism $ r\c A \to B $ of $ \bC $ with a comparison cell $ \rho $ such that:
    \begin{equation} \begin{array}{c}
\rho\c r.a \to b.Tr,   \qq   r\om \te \rho.hA  =  \om r,
\\[5pt]
r\ka \te \rho.mA  =  \rho.Ta \te b.T\rho \te \ka.T^2r,
    \label{5.3.3} \end{array} \end{equation}
    \begin{equation*} \begin{array}{c} 
    \xymatrix  @C=22pt @R=20pt
{
~r~   \ar[r]^-{r\om}   \ar[d]_-{\om r}    &   ~ra.hA~     \ar[d]^-{\rho.hA}  &
~ra.Ta~   \ar[r]^-{\rho.Ta}   \ar[d]_-{r\ka}    &
~b.Tr.Ta~   \ar[r]^-{b.T\rho}    &   ~b.Tb.T^2r~  \ar[d]^-{\ka.T^2r}
\\ 
~b.hB.r~   \ar@{=}[r]    &   ~b.Tr.hA~   &
~ra.mA~    \ar[r]_-{\rho.mA}    &   ~b.Tr.mA~  \ar@{=}[r]    &   ~b.mB.T^2r~
}
    \label{5.3.3bis} \end{array} \end{equation*}

	Given a second colax morphism $ (s, \si)\c (B, b, \om , \ka) \to (C, c, \om , \ka)$, we let
    \begin{equation}
(s, \si).(r, \rho)  =  (sr, \si.Tr \te s\rho),   \q   \si.Tr \te s\rho\c sra \to sb.Tr \to c.T(sr).
    \label{5.3.4} \end{equation}

	This gives a category $ \Cx\bPsa_2(T)$, with identities as above.

\subsection{A double category of pseudo algebras}\label{5.4}
We form now a double category $ \Psa_2(T) $ {\em of pseudo algebras} of $ T$, with lax 
morphisms in horizontal and colax morphisms in vertical. The construction is similar to that 
of $ \Dbl $ in \cite{GP2}, and we shall see that it extends it.

	Its objects are the pseudo $T$-algebras $ \A = (A, a, \om , \ka)$, $\B = (B, a, \om , \ka)$,...; 
its horizontal arrows are the {\em lax} morphisms $ \f = (f, \ph)$, $\g = (g, \ga)$...; its vertical 
arrows are the {\em colax} morphisms $ \r = (r, \rho)$, $\s = (s, \si)$... A cell 
$ \bpi \c (\r \; ^\f _\g \;  \s) $ consists of four morphisms as above together with a cell 
$ \pi \c sf \to gr\c A \to D $ in $ \bC $ (where $ sf $ and $ gr $ are {\em just} morphisms of 
$ \bC$, not of algebras)
%
    \begin{equation} \begin{array}{c} 
    \xymatrix  @C=16pt @R=8pt
{
~\A~   \ar[rr]^-{\f}   \ar[dd]_-{\r}   &&   ~\B~   \ar[dd]^-{\s}  &&&
~A~   \ar[rr]^-{f}   \ar[dd]_-{r}   &   \ar@/^/[ld]^\pi   &   ~B~   \ar[dd]^-{s}
\\ 
&  \bpi  &&&&
\\ 
~\C~   \ar[rr]_-{\g}    &&   ~\D~   &&&
~C~   \ar[rr]_-{g}    &&   ~D~
}
    \label{5.4.1} \end{array} \end{equation}
    \begin{equation*} \begin{array}{cr}
f\c A \to B,   ~~~   \ph\c b.Tf \to f.a,    \q    g\c C \to D,   ~~~   \ga\c d.Tg \to g.c  &
\text{(lax morphisms),}
\\[5pt]
r\c A \to C,   ~~~    \rho\c r.a \to c.Tr,   \q   ~  s\c B \to D,   ~~~~   \si\c s.b \to d.Ts   &
~~~   \text{(colax morphisms),}
\\[5pt]
\pi \c sf \to gr\c A \to D   &    \text{(a 2-cell of  \bC).}
    \label{} \end{array} \end{equation*}

	(Let us note that $ \bpi$ consists of its boundary and a double cell 
$ \pi \c  (r \; ^f _g \;  s) $ of the double category $ \Q(\bC) $ of quintets over the 2-category 
$ \bC$, as displayed in the right diagram above.)

	These data must satisfy the following coherence condition, in the 2-category $ \bC$

\Ndt  (coh)	   $~~    s\ph \te \pi a \te g\rho  =  \si.Tf \te d.T\pi  \te \ga.Tr,$
    \begin{equation} \begin{array}{c} 
    \xymatrix  @C=8pt @R=20pt
{
&   ~TB~   \ar[rr]^-{b}   \arv[rd]|-{\ph}    &&   ~B~   \ar[rd]^-{s}   &&&
&   ~TB~   \ar[rr]^-{b}     \ar[rd]^-{Ts}  &&   ~B~   \ar[rd]^-{s} 
\\ 
~TA~   \ar[ru]^-{Tf}   \ar[rr]|{~a~}   \ar[rd]_-{Tr}  &&  ~A~  \ar[ru]^-{f}    \ar[rd]_-{r}   & \pi  &  ~D~  & = &
~TA~   \ar[ru]^-{Tf}   \ar[rd]_-{Tr}  & T\pi &
~TD~  \ar[rr]|{~d~}     \arv[ru]|-{\si} \arv[rd]|-{\ga}   &&  ~D~
\\ 
&   ~TC~   \ar[rr]_-{c}   \arv[ru]|-{\rho}    &&   ~C~   \ar[ru]_-{g}  &&&
&   ~TC~   \ar[rr]_-{c}   \ar[ru]_-{Tg}    &&   ~C~   \ar[ru]_-{g} 
}
    \label{5.4.2} \end{array} \end{equation}

	The horizontal and vertical composition of double cells are both defined using the 
vertical composition of the 2-category $ \bC$. Namely, for a consistent matrix of double cells
%
    \begin{equation}
    \begin{array}{c}  \xymatrix  @C=12pt @R=8pt  
{
~\iA~  \ar[rr]^-{\f}   \ar[dd]_-{\r}  &&  
~\iB~   \ar[rr]^{\f'}  \ar[dd]^-{\s}  &&  ~\iB'~  \ar[dd]^-{\t}
\\
&  \bpi  &&  \bth
\\
~\iC~  \ar[rr]|{~\g~}  \ar[dd]_-{\r'}  &&  
~\iD~   \ar[rr]|{~\g'~} \ar[dd]^-{\s'}  &&  ~\iD'~ \ar[dd]^-{\t'}  &&&
\\
&  \bze  &&  \bom  &&&&&
\\
~\ic{E}~ \ar[rr]_-{\h}     &&  ~\ic{F}~ \ar[rr]_-{\h'}     &&  ~\ic{F'}~   
}
    \label{5.4.3} \end{array} \end{equation}
we let: 
    \begin{equation}
(\pi\sep \th )  =  \th f \te g'\pi,    \qq    ( \frac{\pi}{\ze} )  =  s'\pi \te \ze r.
    \label{5.4.4} \end{equation}

	We prove below that these cells are indeed coherent. Then, plainly, these composition 
laws are strictly associative and unitary. Moreover, they satisfy the middle-four interchange 
law because this holds in the double category $ \Q(\bC) $ of quintets over the 2-category 
$ \bC$.

	We have thus a forgetful double functor, which is cellwise faithful
    \begin{equation}
U\c \Psa_2(T) \to \Q(\bC),   \q   \A \mapsto A,   \q   \f  \mapsto f,   \q
\r \mapsto r,   \q   \bpi \mapsto \pi.
    \label{5.4.5} \end{equation}

	Finally we verify the axiom (coh) for $(\pi\sep \th )$, which means that
$$	
t(\ph'.Tf \te f'\ph) \te (\th f \te g'\pi)a \te g'g\rho  = \tau.T(f'f) \te d'.T(\th f \te g'\pi) \te (\ga'.Tg \te g'\ga).Tr,
$$
where $a\c TA \to A$ and $d'\c TD' \to D'$ are the structures of the pseudo algeras $\iA$ 
and $\iD'$.

The proof is similar to that of Theorem 
\ref{2.8}. Writing cells as arrows between morphisms, our property amounts to the commutativity 
of the outer diagram below

%
    \begin{equation} \begin{array}{c} 
    \xymatrix  @C=50pt @R=20pt
{
tf'fa~  \ar[r]^-{\th fa}  & ~g'sfa~  \ar[r]^-{g'\pi a}  & ~g'gra  \ar[d]^-{g'g\rho} 
\\ 
tf'b.Tf~   \ar[r]^-{\th b.Tf}  \ar[u]^-{tf'\ph}  &
~g'sb.Tf   \ar[u]_-{g's\ph}     \ar[d]^-{g'\si.Tf}  &  g'gc.Tr
\\   
tb'.T(f'f)   \ar[u]^-{t\ph' Tf}  \ar[d]_-{\tau .T(f'f)}   &  
g'd.T(sf)~  \ar[r]^-{g'd.T\pi}  &  ~g'd.T(gr)  \ar[u]_-{g'\ga.Tr}
\\ 
d'.T(tf'f)~   \ar[r]_-{d'.T(\th f)}   &
~d'.T(g'sf)~  \ar[r]_-{d'.T(g' \pi)}  \ar[u]_-{\ga'T(sf)}  &
~d'.T(g'gr)  \ar[u]_-{\ga'.T(gr)}
}
    \label{5.4.6} \end{array} \end{equation}

\skp

	Here the two hexagons commute by coherence of the double cells $ \bpi $ and $ \bth$, 
and the two rectangles by interchange of 2-cells in $ \bC$.

\subsection{Triple categories of pseudo algebras}\label{5.5}
The double category $ \Psa_2(T) $ can be extended to a triple category $ \iP = \iPs\Psa_2(T) $ 
of pseudo algebras, adding the pseudo morphisms as transversal arrows.

	More precisely, the sets $ P_0, P_1, P_2 $ of arrows of $ \iP $ consist of the pseudo, lax 
and colax morphisms of pseudo algebras, respectively. In dimension 2, the new 01- and 
02-cells are obvious, since $ P_0 $ can be viewed as a subset of $ P_1 $ and $ P_2$. Finally 
a 3-dimensional 012-cell is a `commutative cube' determined by its six faces, as in \eqref{1.2.2}.

	$\iPs\Psa_2(T) $ is a triple category of quintet type, with respect to the obvious 
forgetful functor with values in the triple category $ \trc_3(\iQ(\bC))$. The latter has the 
same objects as $ \bC$, the same arrows in all three directions, quintets for all three kinds 
of double cells and commutative cubes for three dimensional cells.

	Our construction is foreshadowed in Shulman's \cite{Sh}, Remark (4.8). As he points out 
(private communication): a 2-monad on $ \bC $ extends canonically to a triple monad on the 
triple category $ \trc_3(\iQ(\bC))$, and its Eilenberg-Moore object (in the 2-category of strict 
triple categories, strict functors and strict transformations) consists of strict algebras, with strict, lax, and colax morphisms respectively, and higher cells as above.

	In fact, we are more interested in the triple subcategory $ \iPs(\rN\Psa_2(T))$
{\em of normal pseudo algebras}, which - for a convenient 2-monad $ T$, will be proved 
to be equivalent to the triple category $ \iPs(\Dbl)$, obtained by extending the double 
category $ \Dbl $ as above.

	(Taking strict morphisms and double functors in the transversal direction, one would not 
get an equivalence, as will be remarked at the end of the proof of Theorem \ref{5.8}).

\subsection{Graphs of categories. }\label{5.6}
To examine double categories in the present framework, we let $ \bC = \Gph\bCat $ be the 
2-category of graphs $ A = (A_i, \dd^\al) $ in $ \bCat$
    \begin{equation}
\dd^\al\c A_1  \rr  A_0,
    \label{5.6.1} \end{equation}
where a 2-cell $ \ph\c F \to G\c A \to B $ is a pair of natural transformations of ordinary 
functors, consistent with the faces
    \begin{equation}
\ph_i\c F_i \to G_i\c A_i \to B_i,   \q   \dd^\al\ph_1 =  \ph_0\dd^\al   \qq   (i = 0, 1;  \, \al = \pm).
    \label{5.6.2} \end{equation}

	A double category $ \D $ has an underlying graph $ U(\D) = (\Hor_i(\D), \dd^\al)$, formed 
by the category $\Hor_0(\D) $ of objects and horizontal arrows, the category $ \Hor_1(\D) $ of 
vertical arrows and double cells (both with horizontal composition), linked by two ordinary 
functors, the vertical faces $ \dd^\al$.

	This defines a forgetful 2-functor
    \begin{equation}
U\c \bDbl \to \Gph\bCat = \bC,
    \label{5.6.3} \end{equation}
on the 2-category of (small) double categories, double functors and horizontal transformations.

\begin{thm} [Strict double categories as algebras]\label{5.7}
The 2-functor $ U $ defined above is 2-monadic: it gives a comparison 2-isomorphism 
$ K\c \bDbl \to \bAlg(T)$ with the 2-category of T-algebras for the associated 2-monad 
$T = UD$.
\end{thm}
\begin{proof}
A graph of categories $ A = (A_i, \dd^\al) $ generates a free double category $ D(A)$, described as follows.

\Ndt  (a) $ \Hor_0(DA) $ is the category $ A_0$.

\Ndt  (b) $\Ver_0(DA) $ is the free category generated by the graph of sets 
$ \Ob A = (\Ob A_i, \dd^\al)$; its arrows give the vertical arrows $ (u_1,..., u_n) $ of $DA$, 
including the vertical unit $ e(x) $ on an object $ x $ of $ A_0 $ (the {\em empty path} at $ x$).

\Ndt  (c) $ \Ver_1(DA) $ is the free category generated by the graph of sets 
$ \Mor A = (\Mor A_i, \dd^\al)$; its arrows give the double cells $ (a_1,..., a_n) $ of $ DA$, 
including the vertical unit $ e(f) $ on a morphism $ f $ of $ A_0 $

    \begin{equation} \begin{array}{c} 
    \xymatrix  @C=10pt @R=6pt
{
~\bu~   \ar[rr]^-{f_0}   \arDb_-{u_1}    &&   ~\bu~   \arDb^-{v_1}
\\
&   a_1
\\ 
~\bu~   \ar[rr]   \ar@{.}[dd]    &&   ~\bu~    \ar@{.}[dd]         &&&
~x~   \ar[rr]^-{f}   \arDb_-{e(x)}    &&   ~y~   \arDb^-{e(y)}
\\
&&&&&&   e(f)
\\
~\bu~   \ar[rr]   \arDb_-{u_n}    &&   ~\bu~   \arDb^-{v_n}   &&&
~\bu~   \ar[rr]_-{f}    &&   ~\bu~
\\
&   a_n
\\ 
~\bu~   \ar[rr]_-{f_n}    &&   ~\bu~
}
    \label{5.7.1} \end{array} \end{equation}

\ndt  (d) The horizontal composition of these double cells is a concatenation of compositions 
in $ A_1$, which we write as $(a_i\sep b_i)$
$$
((a_1,..., a_n)\sep (b_1,..., b_n))  =  ((a_1\sep b_1),..., (a_n\sep b_n)).
$$
	The obvious embedding $ hA\c A \to UD(A) $ is the 2-universal arrow from $ A $ to 
$ U$. This gives the left 2-adjoint $ D\c \bC \to \bDbl $ and the associated 2-monad 
$ (T, h, m) $ on $ \bC$, with $ T = UD$.

	The comparison $ K $ is plainly an isomorphism of 2-categories.
\end{proof}	

\begin{thm} [Weak double categories as normal pseudo algebras]\label{5.8}
The triple categories $ \iPs(\rN\Psa_2(T)) $ and $ \iPs\Dbl $ are linked by an adjoint 
equivalence of triple categories (\cite{GP10}, Section 5.4)
    \begin{equation} \begin{array}{c}
V\c  \iPs(\rN\Psa_2(T))   \rl    \iPs\Dbl  \cc J,
\\[5pt]
VJ  =  1,  \q   \ep\c JV \iso 1,   \qq   V\ep  =  1,   \q   \ep J  =  1.
    \label{5.8.1} \end{array} \end{equation}
\end{thm}
\begin{proof}
\ndt  (a) A normal pseudo algebra $ \A = (A, c, \ka) $ for $ T $ is a graph of categories 
$ A = (A_i, \dd^\al) $ with an assigned vertical composition of finite paths of vertical arrows 
and cells
    \begin{equation} \begin{array}{ccc}
u_1 \te ... \te u_n  =  c(u_1,..., u_n),  &\q&  a_1 \te ... \te a_n  =  c(a_1,..., a_n),
\\[5pt]
e_x  =  c(e(x)),   &&   e_f  =  c(e(f)),
    \label{5.8.2} \end{array} \end{equation}
and an (invertible) extended associator $ \ka\c c.Dc \to c.mA$.

	As a first consequence of normality, the unary vertical composition is trivial: $ c(u) = u $ 
and $ c(a) = a$, for all items of the category $ A_1$. Second, $ \A $ is trivial in degree 0, in the 
sense that $ (TA)_0 = \Hor_0(DA) = A_0$, while the functor $ c_0 $ and the natural 
transformation $ \ka_0$
    \begin{equation}
c_0\c (TA)_0 \to A_0,   \qq   \ka_0\c c_0.Dc_0 \to c_0.(mA)_0\c (T^2A)_0 \to A_0,
    \label{5.8.3} \end{equation}
are identities: this follows easily from the coherence conditions \eqref{5.1.2}, where 
$ \om A$, $(hA)_0 $ and $ (mA)_0 $ are identities.

\Ndt  (b) $ \A $ can be viewed as an `unbiased' weak double category, where all finite 
vertical compositions are assigned. The fact that $ \ka_0 $ is trivial says that the comparison 
cells of the unbiased associator $ \ka_1 $ are special, i.e.\ their horizontal arrows are identities 
- as required in the axioms of weak double categories, for the binary associator and the unitors.

	The normal pseudo algebra $ \A $ has un underlying weak double category $ V(\A)$, 
obtained by extracting the binary and zeroary vertical operations and their comparisons. 
We get a canonical triple functor $ V$, that sends:

\Ndt  - a pseudo, lax or colax morphism of normal pseudo algebras to the corresponding pseudo, 
lax or colax double functor of weak double categories, reducing the unbiased comparisons of 
finite vertical composition to the `biased ones', of binary and zeroary composition,

\Ndt  - a 2-dimensional cell of type 01 (resp. 02, 12)
%
    \begin{equation} \begin{array}{c} 
    \xymatrix  @C=10pt @R=6pt
{
~\A~   \ar[rr]   \ar[dd]   &   \ar@/^/[ld]^\ph   &   ~\A'~   \ar[dd]
\\ 
\\ 
~\B~   \ar[rr]    &&   ~\B'~
}
    \label{5.8.4} \end{array} \end{equation}
to the corresponding quintet (resp. quintet, generalised quintet) of `functors': the latter are 
reduced to their biased comparisons, but the components of $ \ph $ (on the objects and 
vertical arrows of $ \A$) stay the same,

\Ndt  - a `commutative cube' of 2-dimensional cells to the `commutative cube' of the modified 
cells.

\Ndt  (c) The other way round, we construct a triple functor $ J $ such that $ VJ = 1$, by 
choosing {\em a} `bracketing' of $n$-ary compositions. Namely, a weak double category 
$ \D $ can be extended to a normal pseudo algebra $ J(\D) $ by defining the $n$-ary vertical 
composition (of vertical arrows and cells) as
    \begin{equation}
x_1 \te ... \te x_n  =  (...((x_1 \te x_2) \te x_3)... \, \te x_n),
    \label{5.8.5} \end{equation}
and extending the comparisons. A strict, or lax, or colax double functor becomes a strict, or lax, 
or colax morphism, by extending the comparisons (in the last two cases). For a 2-dimensional 
cell we just note that its coherence with the unbiased comparisons implies coherence with 
the biased ones.

\Ndt  (d) Finally, a normal pseudo algebra $ \A = (A, c, \ka) $ produces an object 
$ JV(\A) = (A, c', \ka') $ with modified unbiased vertical operations and modified unbiased 
comparisons. The identity $ \id(A)$ of the underlying graph extends to an invertible 
{\em pseudo morphism} 
$$ \ep\A\c JV(\A) \to_0 \A $$
satisfying the triangular conditions, as in \eqref{5.8.1}. For a lax morphism $F\c \A \to \B $ 
we get an obvious 01-cell $ \ep F \c JV(F) \to_0 F$, inhabited by an identity; similarly for 
colax morphisms and generalised quintets.

	This point fails if we restrict to the triple categories $ \iS(\rN\Psa_2(T))$ and $\iS\Dbl$ 
(with strict items in the transversal direction), because the component $ \ep \A $ is not a 
strict morphism.
\end{proof}	
		

\end{document}